\documentclass[reqno, a4paper]{amsart}
\pdfoutput=1
\usepackage{amsmath}    
\usepackage{amsthm}     
\usepackage{amssymb}    
\usepackage{amsfonts}
\usepackage{float}
\usepackage{graphics}
\usepackage{graphicx}

 \usepackage{bbm}
 \usepackage{enumerate}
 \usepackage{mathptmx}
 \usepackage{mathrsfs}
 \usepackage{listings}
\usepackage{ulem}
\newtheorem{lemma}{Lemma}[section]
\newtheorem{theorem}{Theorem}[section]
\newtheorem{corollary}{Corollary}[section]

\newtheorem{proposition}{Proposition}[section]

\numberwithin{equation}{section}

\newcommand{\be}{\begin{equation}} \newcommand{\ee}{\end{equation}}
\newcommand{\bd}{\begin{displaymath}} \newcommand{\ed}{\end{displaymath}}
\newcommand{\ba}{\begin{align}} \newcommand{\ea}{\end{align}}
\newcommand{\baa}{\begin{align*}} \newcommand{\eaa}{\end{align*}}
\newcommand{\ben}{\begin{enumerate}} \newcommand{\een}{\end{enumerate}}
\newcommand{\bi}{\begin{itemize}} \newcommand{\ei}{\end{itemize}}

\newcommand{\ud}{\mathrm{d}}

\def\newhbar{{h\mkern-9mu^{-}}}

\def\vec#1{{\bf #1}}

\newcommand{\abs}[1]{\lvert#1\rvert}
\newcommand{\norm}[1]{\|#1\|}

\begin{document}

\title[Convergence of finite volume scheme for Poisson's equation] 
{Convergence of finite volume scheme for three dimensional 
Poisson's equation}
\author[M. Asadzadeh]
{Mohammad Asadzadeh$^{1}$}



\thanks{
 The first author was partially supported by the 
Swedish Research Council (VR) and the 
 Swedish Foundation of Strategic Research (SSF)
in Gothenburg Mathematical Modeling Centre  (GMMC). The second author 
was supported by the Centre for Theoretical Biology at the University of Gothenburg, 
Svenska Institutets \"Ostersj\"osamarbete
scholarship nr. 11142/2013, 
Stiftelsen f\"or Vetenskaplig Forskning och Utbildning i Matematik
(Foundation for Scientific Research and Education in Mathematics), 
Knut and Alice Wallenbergs travel fund, Paul and Marie Berghaus fund, the Royal Swedish 
Academy of Sciences,
and Wilhelm and Martina Lundgrens research fund.}

\address{
$^{1}$
Department of Mathematics,
Chalmers University of Technology and  G\"oteborg University,
SE--412 96, G\"oteborg, Sweden
}
\email{mohammad@chalmers.se}

\author[K. Bartoszek]{Krzysztof Bartoszek 
$^{2}$}
\address{
$^{2}$
Department of Mathematics,
Uppsala University, SE--751 06 Uppsala, Sweden
}
\email{bartoszekkj@gmail.com}




\begin{abstract}
We construct and analyze a finite volume scheme for numerical solution of 
a three-dimensional Poisson equation. 
This is an extension of a two-dimensional approach by S\"uli \cite{Suli:91}. 
Here we derive optimal convergence 
rates in the discrete $H^1$ norm and sub-optimal convergence 
in the maximum norm, where we 
use the maximal available regularity of the exact solution and 
minimal smoothness requirement on the source term. 
We also find a gap in the proof of a key estimate in a reference 
in  \cite{Suli:91} 
for which we present a modified and completed 
 proof. 
Finally, the theoretical results derived in the paper are justified through 
implementing some canonical examples in 3D. 
\end{abstract} 
\maketitle
{\small {\sl Keywords: Finite volume method, Poisson's equation, 
stability estimates, convergence rates}.}

\section{Introduction}\label{intro}

Our motivation for the numerical study of the classical Poisson equation 
stems from its appearance 
in the coupled system of PDEs involving the Vlasov type equations of plasma 
physics with a wide range of application areas, especially
in modelling plasma of 
 Coulomb particles. In this setting 
the common approach has been to consider a continuous Poisson 
solver 
and focus the approximation strategy on the 
study of the associated hyperbolic equations in the system of, e.g. 
Vlasov-Poisson-Fokker-Planck (VPFP) or 
Vlasov-Maxwell-Fokker-Planck (VMFP) equations. 
However, for a system of PDEs involving both 
elliptic and hyperbolic equations, a discrete scheme for the hyperbolic 
equations combined with the continuous solution for the 
elliptic parts requires an unrealistically 
fine degree of resolution for the mesh size of the discretized part. 
Such a combination causes an excessive amount of unnecessary 
computational costs. Indeed, even with 
availability of very fast computational environment, a miss-match will appear 
due to the lack of compatibility 
between the resolution degree for the infinite dimensional 
continuous Poisson solver and a flexible  
numerical scheme for the discretized hyperbolic-type equations in the system. 

The present study concerns numerical approximations of the Poisson equation 
that completes the previous semi-analytic/semi-discrete schemes, 
for the Vlasov-type 
systems, and meanwhile is accurate enough to be comparable with the 
fully discrete numerical schemes  
for the hyperbolic system of PDEs. To this end, We construct and analyze 
a finite volume scheme, prove its stability, 
and derive optimal convergence rates in the discrete 
$H^1$ norm (corresponding to an order of ${\mathcal O(h^2)}$ for 
the exact solution in the Sobolev space $H^2_0(\Omega)$) as well as 
suboptimal convergence rates in the maximum norm 
(the maximum norm estimates are optimal in 2D) for 
the Dirichlet problem for the following 
three dimensional Poisson equation 
\be
\label{Problem}
\left\{
\begin{array}{rccl}
-\nabla^{2}u & = & f & \mathrm{in}~\Omega \\
u & = & 0 & \mathrm{on}~\partial\Omega,
\end{array}
\right. 
\ee
where $\Omega=(0,1)\times (0,1)\times (0,1)$. 

Problem \eqref{Problem} is a simplified version of the general Poisson 
equation formulated as 
\be
\label{Problem0}
\left\{
\begin{array}{rccl}
-\nabla(A\nabla u) & = & f & \mathrm{in}~\Omega, \\
u & = & 0 & \mathrm{on}~\partial\Omega,
\end{array}
\right. 
\ee
where $A$ is a conductivity matrix and  $\Omega$ is a bounded convex domain in 
${\mathbb R}^3$. To simplify  the calculus 
we have assumed that $A=I$ (the identity matrix) and 
considered the cubic Lipschitz domain $\Omega=(0,1)^3$. Note that Problem  
\eqref{Problem0} with a variable coefficient 
matrix $A$ would be much more involved and shift   
our focus away from the study of the 
Poisson operator. 
On the other hand, e.g. 
for a unifying finite element approach for VPFP, 
transferring the Poisson equation 
to a hyperbolic system yields the simple 
but less advantageous problem, (see. e.g. \cite{Arnold_etal}), 
\be
\label{Problem00}
\left\{
\begin{array}{rccl}
v & = & - \nabla u, &  \\
\mbox{ div } v & = &  f.  
\end{array}
\right. 
\ee
Therefore, considering the finite volume method (FVM) for 
the Dirichlet problem \eqref{Problem} 
we can also circumvent such inconvenient issues. 

The convergence results for Problem \eqref{Problem} 
here, considered for a cell-centered finite volume scheme in a 
quasi-uniform mesh, 
may be compared with those of  
a finite element scheme  with no quadrature procedure. 
A finite element scheme combined with a quadrature 
would cause a 
reduced convergence rate by an order of $\sim {\mathcal O} (h^{1/2})$. 
In this aspect, compared to standard finite elements, 
the usual finite volume method (as the finite difference) 
is quasi-optimal. 

The main advantage of the 
finite volume method is its  
{\sl local conservativity property} for the numerical flux. This property makes 
the finite volume method an attractive tool for approximating model problems 
emphasizing the flux, e.g.  as in the case 
of some hyperbolic PDEs describing fluid problems and 
conservation laws, see \cite{Eymard_etal:2003} 
for further details. A draw-back in FVM formulation is that,  
in higher dimensions, in addition to the expected 
theoretical challenges, the calculus is  
{\sl seemingly involved} and yields a rather lengthy and tedious 
representation. Despite this fact, 
the finite volume method has been studied for both the 
Poisson equation, fluid 
problems and other PDEs by several authors in various settings: 
e.g. the discontinuous finite volume method for second-order elliptic problems 
in two-dimensions is considered in \cite{Bi_etal:2010}, where the closeness 
of the FVM to the interior penalty method is demonstrated 
and optimal error estimates are derived in $L_2$- and $L_\infty$-norms. 
A three dimensional discrete duality finite volume scheme for nonlinear 
elliptic equations is studied in \cite{Coudiere_etal:2011}, where 
well-posedness and a priori $L_p$-error bounds  
 are discussed. These are 
$L_p$ convergence analysis with no particular consideration of 
their optimality.  A more computation oriented, second-order finite volume 
scheme in three dimensions: \cite{Wang_etal:2006}, deals with 
computing eigenvalues of a Schr\"odinger type operator.  
As another computational exposition: in 
\cite{Oevermann_etal:2009} the authors construct 
a shape interface FVM for elliptic equations on Cartesian grids 
in three dimensions with second order accuracy in 
$L_2$- and $L_\infty$-norms. 
The authors consider also variable coefficients based on 
using a particular piecewise trilinear ansatz. 
As for the fluid problems, a 3D finite volume scheme 
is  presented for the 
ideal magneto-hydrodynamics in \cite{Arminjon_etal}.  
Some theoretical analysis for the 
upwind FVM on the counter-example of Peterson, 
for a two-dimensional, time dependent advection problem, 
can be found in \cite{Bouche_etal:2010}. For a detailed study 
of the finite volume method for a compressible flow see \cite{Novotny:2004}. 

The most relevant works for our study 
are some results 
by S\"uli et. al. , e.g.  \cite{Suli:91}, for 
a two-dimensional version of our work, 
and \cite{Suli:92} and \cite{Suli_etal:97}, considering 
the accuracy of cell-vertex FVM for time-dependent advection- 
and convection-diffusion problems, 
respectively.  Finally, 
a thorough theoretical 
study for the numerical solutions of general, linear, nonlinear and quasilinear 
 elliptic problems are given by B\"ohmer in \cite{Bohmer:2010}, where 
most numerical methods are rigorously featured.

Below,  for the sake of completeness, 
we recall some classical results 
concerning the regularities connecting the solution  and 
the data for Problem \eqref{Problem} in different geometries. 
First we state these results in  ${\mathbb R}^n$ and then for 
an open set $\Omega\subset {\mathbb R}^n$ with smooth boundary. 
For details we refer the reader to, e.g. Folland \cite{Folland:76}. 
In Propositions \ref{ThmFolland1}-\ref{ThmFolland_Sobolev3} below,  
$\Omega$ 
is assumed to have a smooth boundary. 

\begin{proposition}\label{ThmFolland1}
Suppose $f\in L_1({\mathbb R}^n)$, and also that 
$\int_{\vert{x}\vert>1}\vert{f(x)}\vert\log \vert{x}\vert\, dx <\infty$ 
for $n=2$. Let $N$ be the fundamental solution 
of the $-\nabla^{2}$ operator:  $-\nabla^{2}N=\delta$.
Then $u=f*N$ is locally integrable and is a distribution solution for  
$-\nabla^{2}u = f$.
\end{proposition}
\begin{proposition}\label{ThmFolland2} 
If $f$ satisfies the conditions of Proposition \ref{ThmFolland1} and 
in addition $f$ is ${\mathcal C}^\alpha(\Omega)$ for some $\alpha\in (0,1)$ 
on some open set $\Omega$, then 
$u=f*N$ is  ${\mathcal C}^{2+\alpha}$ on $\Omega$. 
\end{proposition}

\begin{corollary}\label{BersJohnSchechter} 
If $f\in {\mathcal C}^{k+\alpha}(\Omega)$ for some integer $k$ and 
$\alpha\in (0,1)$ 
then $u\in {\mathcal C}^{k+2+\alpha}(\Omega)$. 
\end{corollary}

To express in, $L_2$-based,  
Sobolev spaces (see Adams\cite{Adams:75} for details) we have 
\begin{proposition}\label{ThmFolland_Sobolev3} If $f\in H^k(\Omega)$ then 
$u\in H^1_0(\Omega)\cap H^{2+k}(\Omega)$. 
\end{proposition}

 For a general bounded convex domain $\Omega$, by Dirichlet principle, given 
$f\in H^{-1}(\Omega)$, there exists a unique solution, 
$u\in H^1_0(\Omega)$, satisfying 
\eqref{Problem}, and the mapping $f\longmapsto u$ is a Hilbert space 
isomorphism from $H^{-1}(\Omega)$ onto $H^1_0(\Omega)$. 
This is crucial in our study where, 
in order to derive optimal convergence with minimum smoothness 
requirement on the exact solution, 
we shall assume the data $f$ to belong to $H^{-1}$, i.e. 
 the dual of $H^1_0(\Omega)$. Then 
for 
$f\in H^{\sigma}(\Omega)$, we have $u\in H^{\sigma+2}(\Omega)$
where $-1\le \sigma <1.$ To justify the 
regularity preserving property we refer the reader to studies based on Green's 
function approaches,  e.g. in 
\cite{Fromm:93} and \cite{Mazya.Rossmann:2010}. 

The purpose of this study is 
to generalize the two dimensional 
results in \cite{Suli:91} from the rectangular domain 
$\Omega=(0,1)\times (0,1)$ to 
the cubic domain $\Omega=(0,1)\times (0,1) \times (0,1)$. 
The study of the finite volume scheme in three dimensions is somewhat 
different from a straightforward generalization of the two dimensional case and 
there are adjustments that need to be made for the dimension. 
We also provide a corrected (cf. \cite{Drazic:86}) proof of
Theorem 4.2 (in \cite{Suli:91}) utilized for the convergence of the 
finite volume method. 

For Problem \eqref{Problem}, existence, uniqueness, and regularity 
studies are extensions of two-dimensional results in \cite{Grisvard:85}:  
$f\in H^{-1}(\Omega)$ implies that: there exists a unique solution 
\mbox{
$ u\in H_0^{1}(\Omega)$,}
and for 
$f\in H^{s}(\Omega)$, with   $-1\le s<1,\,\, s\neq \pm 1/2,$  
$\,\, u\in H^{s+2}(\Omega)$. 
The finite volume scheme can be described as:  
exploiting divergence from the differential equation  \eqref{Problem}
integrating over disjoint {}''volumes'' 
and using Gauss' divergence theorem to convert volume-integrals to 
surface-integrals, and then discretizing to obtain the 
approximate solution 
$u_h$, with $h$ denoting the mesh size. 
Here, the finite volume method is defined on the Cartesian product 
of non-uniform meshes as a Petrov-Galerkin method using 
 piecewise trilinear trial functions on a {\sl finite element} mesh and 
piecewise constant test functions on the dual box mesh. The main result 
of this paper: Theorem \ref{FVTH1}, together with the optimal 
finite element estimate in Theorem \ref{FETH2}, justifies the 
sharpness of our estimate in $L_2$. The $L_\infty$ estimate in 
three dimensions is suboptimal. 

\begin{theorem}\label{FVTH1}
 The finite volume error estimates for 
general non-uniform and quasi-uniform meshes in 
$\Omega\subset {\mathbb R}^d,\, d=2,3$, 
are given by 
 \begin{equation}
 \|{u-u_h}\|_{1,h}\le C
h^s|u|_{H^{s+1}},\qquad 
 \|{u-u_h}\|_{\infty}\le C
h^{s+1-d/2}|\log h||u|_{H^{s+1}}, \quad 1/2<s\le 2.
\end{equation}
\end{theorem}
whereas the corresponding finite element estimates can be read as:  
\begin{theorem}\label{FETH2} (cf \cite{Lin.Thomee.Wahlbin:91}) 
\\
a) For the finite element solution of the Poisson problem \eqref{Problem}, 
in two dimensions, with a 
quasiuniform triangulation we have the error estimate: 
$$
\|{u-u_h}\|_{1,\infty}\le C
h^{r}|\log h|\times \|u\|_{r+1,\infty},\qquad r\le 2
$$
b) 
$\forall \varepsilon\in (0,1)\,\,$  small, $\,\, \exists \, \,C_{\varepsilon}$ 
such that 
$
\|{u-u_h}\|_{1,\infty}\ge C_{\varepsilon}h^{r-\varepsilon}|\log h|. 
$
\end{theorem}
Note that, in the two dimensional case, $s=2$ in Theorem \ref{FVTH1} 
corresponds to $r=1$ in Theorem \ref{FETH2}, 
whereas the optimal $L_\infty$ estimate in 2D is not generalized to the 3D case.

\section{The finite volume method in 3D}
A version of the three--dimensional scheme construction has also 
been discussed in \cite{Asadzadeh.Bartoszek:2011}.
On our spatial domain $\Omega$ we construct an arbitrary 
(not necessarily uniform) mesh 
$\bar{\Omega}^{h}=\bar{\Omega}^{h}_{x} \times\bar{\Omega}^{h}_{y} \times\bar{\Omega}^{h}_{z}$
as a Cartesian product of three one--dimensional meshes,
\bd
\begin{array}{rcl}
\bar{\Omega}^{h}_{x} & =& \{x_{i},~ i=0,\ldots,M_{x}:~x_{0}=0,~x_{i}-x_{i-1}=h^{x}_{i},~x_{M_{x}}=1 \} \\
\bar{\Omega}^{h}_{y} & =& \{y_{j},~ j=0,\ldots,M_{y}:~x_{0}=0,~y_{j}-y_{j-1}=h^{y}_{j},~y_{M_{y}}=1 \} \\
\bar{\Omega}^{h}_{z} & =& \{z_{k},~ k=0,\ldots,M_{z}:~x_{0}=0,~z_{k}-z_{k-1}=h^{z}_{k},~z_{M_{z}}=1 \}.
\end{array}
\ed
We further define 
$\Omega^{h}_{x}:=\bar{\Omega}^{h}_{x} \cap (0,1]$, $\Omega^{h}_{y}:=\bar{\Omega}^{h}_{y} \cap (0,1]$,
$\Omega^{h}_{z}:=\bar{\Omega}^{h}_{z} \cap (0,1]$,
$\partial\Omega^{h}_{x}:=\{0,1\}\times\Omega^{h}_{y} \times\Omega^{h}_{z} $,
$\partial\Omega^{h}_{y}:= \Omega^{h}_{x} \times \{0,1\}\times\Omega^{h}_{z} $,
$\partial\Omega^{h}_{z}:= \Omega^{h}_{x} \times\Omega^{h}_{y}\times \{0,1\} $,
$\Omega^{h}:=\Omega \cap \bar{\Omega}^{h}$ and $\partial\Omega^{h}:=\partial\Omega \cap \bar{\Omega}^{h}$.
With each mesh point $(x_{i},y_{j},z_{k}) \in \Omega^{h}$ 
we associate the finite volume element 
$$ \omega_{ijk} := (x_{i-1/2},x_{i+1/2}) \times(y_{j-1/2},y_{j+1/2}) 
\times(z_{k-1/2},z_{k+1/2}),$$
where
\bd
\begin{array}{rclrcl}
x_{i-1/2} & :=& x_{i}-\frac{h^{x}_{i}}{2}, \quad & x_{i+1/2} & :=& x_{i}+\frac{h^{x}_{i+1}}{2}, \\
y_{j-1/2} & :=& y_{j}-\frac{h^{y}_{j}}{2}, & y_{j+1/2} & :=& y_{j}+\frac{h^{y}_{j+1}}{2}, \\
z_{k-1/2} & :=& z_{k}-\frac{h^{z}_{k}}{2}, & z_{k+1/2} & :=& z_{k}+\frac{h^{z}_{k+1}}{2}, 
\end{array}
\ed
and denote the dimensions of the volume element $\omega_{ijk}$ by,
$$ \newhbar^{x}_{i}:= \frac{h^{x}_{i}+h^{x}_{i+1}}{2}, \quad \newhbar^{y}_{j}:= \frac{h^{y}_{j}+h^{y}_{j+1}}{2},\quad \newhbar^{z}_{k}:= \frac{h^{z}_{k}+h^{z}_{k+1}}{2}, $$
see Fig. \ref{figomegaijkbox}.
\begin{figure}[ht!]
\begin{center}
\includegraphics[width=0.5\textwidth]{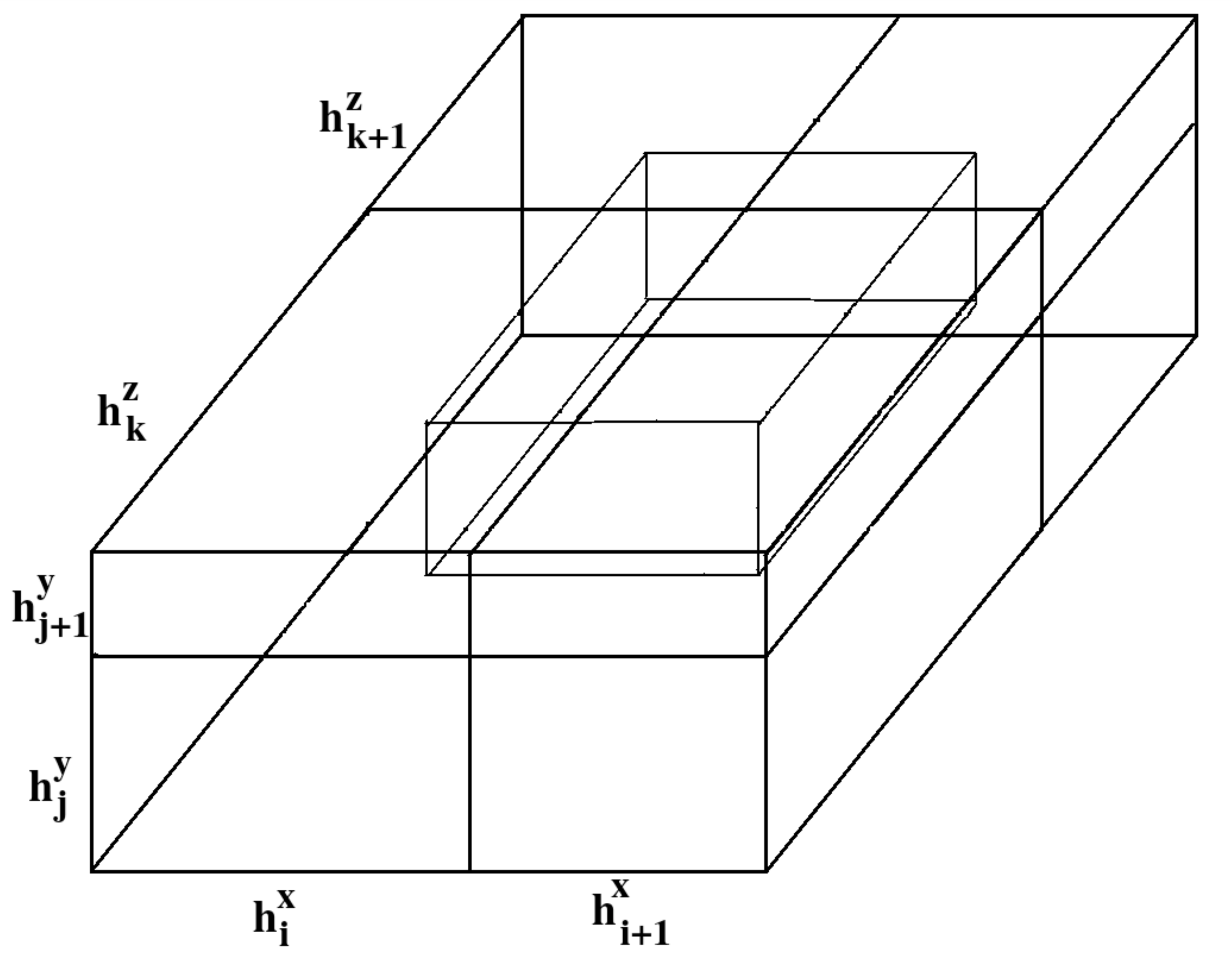}
\caption{
Part of mesh showing the grid and finite volume $\omega_{ijk}$ (inside box) in three dimensions.
}\label{figomegaijkbox} 
\end{center}
\end{figure}

The characteristic function of the box $\omega_{ijk}$, i.e. $\chi_{ijk}$  
belongs to $H^{\tau}(\mathbb{R}^{3})$ for all 
\mbox{
$\tau <1/2$.} 
This can be easily verified by the fact that the Fourier 
transform of the characteristic function of the unit interval 
$\chi_{(0,1)}$ is the {\sl sinc function}: $\sin \xi/\xi$. Thus using the 
Fourier transform we may determine the Sobolev class of $\chi_{ijk}$.  
To this end, for each $s\in {\mathbb R}^+$  we recall the 
operator $\Lambda^s$ defined as  
$(\Lambda^s \xi)\,\hat{}\,( \chi)=(1+\abs{\xi}^2)^{s/2}\hat \chi(\xi)$ 
and the Sobolev norm of order $s$, 
\begin{equation}
\begin{split}
\norm{\chi_{ijk}}_s^2& = 
\norm{\Lambda^s \chi_{ijk}}_{L_2({\mathbb R}^3)}^2\\
&=\int_{{\mathbb R}^3}(1+\abs{\xi_1}^2+\abs{\xi_2}^2 + \abs{\xi_3}^2)^s 
\Big(\frac{\sin\xi_1}{\xi_1}\Big)^2\cdot 
\Big(\frac{\sin\xi_2}{\xi_2}\Big)^2\cdot 
\Big(\frac{\sin\xi_3}{\xi_3}\Big)^2\, d\xi.
\end{split}
\end{equation}
We split the above integral as 
$$
\int_{{\mathbb R}^3}\bullet\,\,d\xi =\int_{\abs{\xi}\le 1}\bullet\,\,d\xi +\int_{\abs{\xi}> 1}\bullet\,\,d\xi, 
$$
and check for which $s$-values the integrals on the right hand side 
 converge. For the first integral, since 
$\lim_{\xi_i\to 0} {\sin\xi_i}/{\xi_1}=1, \,\, i=1, 2, 3,$  
we get an immediate bound. As for the second integral 
we have that, 
\begin{equation}
\begin{split}
\int_{\abs{\xi}> 1}  (1+&\abs{\xi_1}^2+\abs{\xi_2}^2 + \abs{\xi_3}^2)^s 
\Big(\frac{\sin\xi_1}{\xi_1}\Big)^2\cdot 
\Big(\frac{\sin\xi_2}{\xi_2}\Big)^2\cdot 
\Big(\frac{\sin\xi_3}{\xi_3}\Big)^2\,\,d\xi \\
&\le \int_{\abs{\xi}> 1} (1+\abs{\xi_1}^2)^s(1+\abs{\xi_2}^2)^s(1+\abs{\xi_3}^2)^s
\frac 1{\abs{\xi_1^2}}\cdot\frac 1{\abs{\xi_2^2}}\cdot \frac 1{\abs{\xi_3^2}}\,\,d\xi
\\
&\le \prod_{j=1}^3\int_{\abs{\xi_j}> r_j}  (1+\abs{\xi_j}^2)^s\frac 1{\abs{\xi_j^2}}\,\,d\xi
=\prod_{j=1}^3\int_{r_j}^\infty(1+r^2)^s\frac 1{{r^2}}\,  dr, 
\end{split}
\end{equation}
which converges for $2s-2<-1$, i.e. $s<1/2$. 
Since $\chi\in H^{\tau}(\Omega),\,\, \tau<1/2$, we may 
 assume that $f\in H^{\sigma}(\Omega)$ for $\sigma>-1/2$.  
Then the convolution $\chi_{ijk}\ast f$
will be continuous on $\mathbb{R}^{3}$ and if we have $f \in L_{loc}^{1}(\Omega)$, 
then 
\be
\frac{-2}{\vert \omega_{ijk} \vert} 
\int\limits_{\partial \omega_{ijk}} \frac{\partial u}{\partial \mathbbm {n}} \ud s
= 
\frac{1}{\vert \omega_{ijk} \vert} 
\left(\chi_{ijk} \ast f \right) \left(x_{i},y_{j},z_{k} \right), 
\ee
where $\vert \omega_{ijk}\vert = \newhbar^{x}_{i}\newhbar^{y}_{j}\newhbar^{z}_{k}.$
Let now $S^{h}_{0}$ be the set of piecewise continuous 
trilinear functions 
defined on the cubic rectangular partition of $\bar{\Omega}$ induced by $
\bar{\Omega}^{h}$ and 
vanishing on $\partial \Omega$.
We can now construct the finite volume approximation 
$u^{h} \in S_{0}^{h}$ of $u$ as satisfying,
\be
\label{eqfvm}
\frac{-2}{\newhbar^{x}_{i}\newhbar^{y}_{j}\newhbar^{z}_{k}} 
\int\limits_{\partial \omega_{ijk}} 
\frac{\partial u^{h}}{\partial \mathbbm{n}} \ud s
=
\frac{1}{\newhbar^{x}_{i}\newhbar^{y}_{j}\newhbar^{z}_{k}} 
\left(\chi_{ijk} \ast f \right)
(x_{i},y_{j},z_{k})\qquad \mathrm{for}\quad (x_{i},y_{j},z_{k})\in\Omega^{h}.
\ee
Here the factor $2$ appears due to the jump of $\chi_{ijk}$ across the 
inter-element 
boundaries on $\partial \omega_{ijk}$, and will  
not matter for any of the stability results and convergence rates 
as considered by \cite{Asadzadeh.Bartoszek:2011} and \cite{Suli:91}  
but only in numerical implementations of the scheme.

\section{Properties of the scheme and stability estimates}
To investigate the behavior of this scheme we will rewrite it as a 
finite difference scheme. To this end, we define
the averaging operators 
(all are presented, since 
due to miss-matches in indexing discrepancies these operators are not presentable 
 in a single generic form) 
\be
\label{eqsAvOp}
\begin{array}{ll}
\mu_{xy}u_{ijk} &:=  \frac{1}{16\newhbar^{x}_{i}\newhbar^{y}_{j}} \left(h^{x}_{i}h^{y}_{j}u_{i-1,j-1,k}+h^{x}_{i+1}h^{y}_{j}u_{i+1,j-1,k}+12\newhbar^{x}_{i}\newhbar^{y}_{j}u_{ijk}
\right.\\ &\qquad \qquad \left.
+h^{x}_{i}h^{y}_{j+1}u_{i-1,j+1,k}+h^{x}_{i+1}h^{y}_{j+1}u_{i+1,j+1,k} \right), \\ 
\mu_{xz}u_{ijk} & :=  \frac{1}{16\newhbar^{x}_{i}\newhbar^{z}_{k}} \left(h^{x}_{i}h^{z}_{k}u_{i-1,j,k-1}+h^{x}_{i+1}h^{z}_{k}u_{i+1,j,k-1}+12\newhbar^{x}_{i}\newhbar^{z}_{k}u_{ijk}
\right.\\  & \qquad \qquad \left.
+h^{x}_{i}h^{z}_{k+1}u_{i-1,j,k+1}+h^{x}_{i+1}h^{z}_{k+1}u_{i+1,j,k+1} \right), \\ 
\mu_{yz}u_{ijk} & :=  \frac{1}{16\newhbar^{y}_{j}\newhbar^{z}_{k}} \left(h^{y}_{j}h^{z}_{k}u_{i,j-1,k-1}+h^{y}_{j+1}h^{z}_{k}u_{i,j+1,k-1}+12\newhbar^{y}_{j}\newhbar^{z}_{k}u_{ijk}
\right.\\ & \qquad \qquad \left.
+h^{y}_{j}h^{z}_{k+1}u_{i,j-1,k+1}+h^{y}_{j+1}h^{z}_{k+1}u_{i,j+1,k+1} \right), 
\end{array}
\ee
and the divided differences,
\bd
\begin{array}{rclrcl}
\Delta^{-}_{x}u_{i,j,k} & =& \frac{u_{i,j,k}-u_{i-1,j,k}}{h^{x}_{i}}, & 
\qquad \Delta^{+}_{x}u_{i,j,k} & =& \frac{u_{i+1,j,k}-u_{i,j,k}}{\newhbar^{x}_{i}},\\
\Delta^{-}_{y}u_{i,j,k} & =& \frac{u_{i,j,k}-u_{i,j-1,k}}{h^{y}_{j}}, & 
\qquad \Delta^{+}_{y}u_{i,j,k} & =& \frac{u_{i,j+1,k}-u_{i,j,k}}{\newhbar^{y}_{j}},\\
\Delta^{-}_{z}u_{i,j,k} & =& \frac{u_{i,j,k}-u_{i,j,k-1}}{h^{z}_{k}}, & 
\qquad \Delta^{+}_{z}u_{i,j,k} & =& \frac{u_{i,j,k+1}-u_{i,j,k}}{\newhbar^{z}_{k}}.
\end{array}
\ed 
Then, we can write
\bd
\newhbar^{x}_{i}\newhbar^{y}_{j}\newhbar^{z}_{k}\left(\Delta^{+}_{x}\Delta^{-}_{x}\mu_{yz} 
+\Delta^{+}_{y}\Delta^{-}_{y}\mu_{xz} + \Delta^{+}_{z}\Delta^{-}_{z}\mu_{xy}
 \right)u_{i,j,k} = \int\limits_{\partial \omega_{ijk}}
\frac{\partial u}{\partial \mathbbm{n}}\ud s.
\ed
This allows us to restate the finite volume scheme, \eqref{eqfvm} 
as the following finite difference scheme,
\be\label{eqfds}
\begin{array}{rcll}
-2\left(\Delta^{+}_{x}\Delta^{-}_{x}\mu_{yz} +\Delta^{+}_{y}\Delta^{-}_{y}\mu_{xz} + \Delta^{+}_{z}\Delta^{-}_{z}\mu_{xy} \right)u^{h} & =& T_{111}f 
\quad & ~\mathrm{in}\,\,\Omega^{h}, \\
u^{h} & =& 0 \quad &~\mathrm{on}\,\,\partial\Omega^{h},
\end{array}
\ee
where
\bd
\left(T_{111}f \right)_{ijk} = \frac{1}{\newhbar^{x}_{i}\newhbar^{y}_{j}\newhbar^{z}_{k}}\left(\chi_{ijk} \ast f\right)(x_{i},y_{j},z_{k}).
\ed
To extend \eqref{eqfds} to higher than three dimensions, 
the same scheme will apply, however the definition of $\mu$ will change.
If we look at carefully how this averaging operator works, it appears that 
the main 
difference will be what will correspond to the factor 
$12$ appearing as the coefficient of the central term in 
 \eqref{eqsAvOp}. In fact
if we denote by $d$ the dimension then, 
\be\label{FinDiffCentral} 
\begin{array}{ll}
&\mu_{x_{1}x_{2}\ldots x_{d-1}} u_{i_{1}\ldots i_{d}} =
\frac{1}{2^{d+1}}\frac{1}{\newhbar^{x_{1}}_{i_{1}}\ldots \newhbar^{x_{d-1}}_{i_{d-1}}}\left(
3\cdot2^{d-1}\cdot\newhbar^{x_{1}}_{i_{1}}\ldots \newhbar^{x_{d-1}}_{i_{d-1}}u_{i_{1}\ldots i_{d}}
\right. \\ & \qquad \left.  +
h^{x_{1}}_{i_{1}}\ldots h^{x_{d-1}}_{i_{d-1}} u_{i_{1}-1,\ldots, i_{d-1}-1,i_{d}} + \ldots
+  h^{x_{1}}_{i_{1}+1}\ldots h^{x_{d-1}}_{i_{d-1}+1} u_{i_{1}+1,\ldots, i_{d-1}+1,i_{d}}
\right).
\end{array} 
\ee
We will study the behavior of the scheme defined by \eqref{eqfds} in the discrete $H^{1}$ norm $\Vert \cdot \Vert_{1,h},$
\bd
\Vert v \Vert_{1,h} = \sqrt{\Vert v \Vert^{2}+\vert v \vert_{1,h}^{2}},
\ed
where $\Vert \cdot \Vert$ is the discrete $L_{2}$-norm over $\Omega^{h}$ 
(we suppressed $h$ in the discrete $L_2$), i.e., 
\bd
\begin{array}{rclrcl}
\Vert v \Vert&=&\sqrt{(v,v)}, &\qquad  (v,w)=\sum\limits_{i=1}^{M^{x}-1}\sum\limits_{j=1}^{M^{y}-1}\sum\limits_{z=1}^{M^{z}-1}\newhbar^{x}_{i}\newhbar^{y}_{j}\newhbar^{z}_{k}v_{i,j,k}w_{i,j,k}, 
\end{array}
\ed
and $\vert \cdot \vert_{1,h}$ is the discrete $H^{1}$-seminorm given by 
\bd
\vert v \vert_{1,h} = \sqrt{\Vert \Delta^{-}_{x} v \vert ]_{x}^{2}+\Vert \Delta^{-}_{y} v \vert ]_{y}^{2}+\Vert \Delta^{-}_{z} v \vert ]_{z}^{2}},
\ed
with 
\bd
\begin{array}{llll}
\Vert v \vert]_{x}^{2} & =  (v,v]_{x}, & \qquad (v,w]_{x} & =   \sum\limits_{i=1}^{M^{x}} \sum\limits_{j=1}^{M^{y}-1} \sum\limits_{k=1}^{M^{z}-1}h^{x}_{i}\newhbar^{y}_{j}\newhbar^{z}_{k}v_{i,j,k}w_{i,j,k},\\
\Vert v \vert]_{y}^{2} & =  (v,v]_{y}, & \qquad (v,w]_{y} & =   \sum\limits_{i=1}^{M^{x}-1} \sum\limits_{j=1}^{M^{y}} \sum\limits_{k=1}^{M^{z}-1}\newhbar^{x}_{i}h^{y}_{j}\newhbar^{z}_{k}v_{i,j,k}w_{i,j,k},\\
\Vert v \vert]_{z}^{2} & =  (v,v]_{z}, & \qquad (v,w]_{z} & =   \sum\limits_{i=1}^{M^{x}-1} \sum\limits_{j=1}^{M^{y}-1} \sum\limits_{k=1}^{M^{z}}\newhbar^{x}_{i}\newhbar^{y}_{j}h^{z}_{k}v_{i,j,k}w_{i,j,k}.
\end{array}
\ed
In addition we define the discrete $H^{-1}$ norm as,
\bd
\Vert v \Vert_{-1,h} = \sup_{w\in H^{1,h}_0(\bar{\Omega}^{h})}\frac{\vert (v,w) \vert}{\Vert w \Vert_{1,h}},
\ed
where the supremum is taken over all non--zero mesh functions on 
$\bar{\Omega}^{h}$ vanishing on $\partial\bar{\Omega}^{h}$.

We will now state and prove two {\sl coercivity-type} estimates 
 describing relationships between 
the above and our operators.
These are essentially the same as Lemmas 3.1 and 3.2 in 
\cite{Suli:91} with the coefficients
adjusted for the three dimensional case.  

\begin{lemma}\label{lem31}
Let $v$ be a mesh function on $\bar{\Omega}^{h}$.
 If $v=0$ on $\partial \Omega^{h}_{\alpha\beta}$, then
$(\mu_{\alpha\beta}v,v]_{\gamma} \ge \frac{5}{8} \Vert v \vert]_{\gamma}^{2}, $ 
in the following three cases: 

(i) \quad$\alpha\beta:=xy,\,\,\, \gamma:=z$, \qquad 
(ii) \quad $\alpha\beta:=xz,\,\,\, \gamma:=y$, \qquad and \,\,\,
(iii) \quad $\alpha\beta:=yz,\,\,\, \gamma:=x$.  






\end{lemma}
\begin{proof}
We give a proof for $i)$ here, as both $ii)$ and  $iii)$ 
will be obtained by the same way.
Note, in particular,  that $v=0$ on $\partial \Omega^{h}_{xy}$, and we shall 
also use 
$a^{2}/2+2ab+b^{2}/2 \ge -a^{2}/2-b^{2}/2$. 
To proceed let 
\bd
\begin{array}{ll}
{\mathcal A}_1:=\sum\limits_{i=1}^{M^{x}-1} &
\sum\limits_{j=1}^{M^{y}-1}
 \left(
\newhbar^{x}_{i}\newhbar^{y}_{j}v_{ijk}^{2} +
h^{x}_{i}h^{y}_{j}v_{i-1,j-1,k}v_{ijk} +
h^{x}_{i+1}h^{y}_{j}v_{i+1,j-1,k}v_{ijk} \right.\\ 
&\qquad  \left.
+h^{x}_{i}h^{y}_{j+1}v_{i-1,j+1,k}v_{ijk} 
+
h^{x}_{i+1}h^{y}_{j+j}v_{i+1,j+1,k}v_{ijk} \right).
\end{array}
\ed
Then, we use the shift law, vanishing boundary conditions,  
and split the terms in ${\mathcal A}_1$ at the end-point 
indices to obtain 
\bd
\begin{array}{ll} 
{\mathcal A}_1=&
\sum\limits_{i=2}^{M^{x}-1}\,\,\sum\limits_{j=2}^{M^{y}-1}  h^{x}_{i}h^{y}_{j}\left(v_{i-1,j-1,k}+v_{ijk} \right)v_{ijk}
+\sum\limits_{j=1}^{M^{y}-1}(h^{x}_{1}h^{y}_{j}v_{1jk}^{2})
+\sum\limits_{i=1}^{M^{x}-1}(h^{x}_{i}h^{y}_{1}v_{i1k}^{2})\\
& +\sum\limits_{i=1}^{M^{x}-2}\,\, \sum\limits_{j=2}^{M^{y}-1} h^{x}_{i+1}h^{y}_{j}\left(v_{i+1,j-1,k}+v_{ijk} \right)v_{ijk}
+\sum\limits_{j=1}^{M^{y}-1}(h^{x}_{M^{x}}h^{y}_{j}v_{M^{x}jk}^{2})\\
& +\sum\limits_{i=1}^{M^{x}-2}(h^{x}_{i+1}h^{y}_{1}v_{i1k}^{2})
+\sum\limits_{i=2}^{M^{x}-1}\,\,\sum\limits_{j=1}^{M^{y}-2}h^{x}_{i}h^{y}_{j+1}\left(v_{i-1,j+1,k}+v_{ijk} \right)v_{ijk}\\
&+\sum\limits_{j=1}^{M^{y}-1}(h^{x}_{1}h^{y}_{j}v_{1jk}^{2} )
  +\sum\limits_{i=2}^{M^{x}-1}(h^{x}_{i}h^{y}_{M^{y}}v_{iM^{y}k}^{2}) 
\\
&+\sum\limits_{i=1}^{M^{x}-2}\,\,
\sum\limits_{j=2}^{M^{y}-2}h^{x}_{i+1}h^{y}_{j+1}\left(v_{i+1,j+1,k}+v_{ijk} \right)v_{ijk}
+\sum\limits_{j=1}^{M^{y}-1}(h^{x}_{M^{x}}h^{y}_{j+1}v_{M^{x}jk}^{2}) \\
&\qquad \qquad
+\sum\limits_{i=1}^{M^{x}-2}(h^{x}_{i+1}h^{y}_{M^{y}}v_{iM^{y}k}^{2}). 
\end{array}
\ed
The single sums in the above identity are 
all nonnegative, removing them it follows that 
\bd
\begin{array}{ll}
{\mathcal A}_1 &\ge \sum\limits_{i=2}^{M^{x}-1}\,
\sum\limits_{j=2}^{M^{y}-1}h^{x}_{i}h^{y}_{j}v_{i-1,j-1,k}v_{ijk}
+\sum\limits_{i=2}^{M^{x}-1}\,
\sum\limits_{j=2}^{M^{y}-1}h^{x}_{i}h^{y}_{j}v_{ijk}^{2} 
+\sum\limits_{i=1}^{M^{x}-2}\,
\sum\limits_{j=2}^{M^{y}-1}h^{x}_{i+1}h^{y}_{j}v_{i+1,j-1,k}v_{ijk}\\
&+\sum\limits_{i=1}^{M^{x}-2}\,\,
\sum\limits_{j=2}^{M^{y}-1}h^{x}_{i+1}h^{y}_{j}v_{ijk}^{2} 
+\sum\limits_{i=2}^{M^{x}-1}\,\, 
\sum\limits_{j=1}^{M^{y}-2}h^{x}_{i}h^{y}_{j+1}v_{i-1,j+1,k}v_{ijk}
+\sum\limits_{i=2}^{M^{x}-1}\,\,
\sum\limits_{j=1}^{M^{y}-2}h^{x}_{i}h^{y}_{j+1}v_{ijk}^{2}\\
& 
+\sum\limits_{i=1}^{M^{x}-2}\,\,
\sum\limits_{j=1}^{M^{y}-2}h^{x}_{i+1}h^{y}_{j+1}v_{i+1,j+1,k}v_{ijk}
+\sum\limits_{i=1}^{M^{x}-2}\,\,
\sum\limits_{j=1}^{M^{y}-2}h^{x}_{i+1}h^{y}_{j+1}v_{ijk}^{2}=: {\mathcal B}_1. 
\end{array}
\ed
For simplicity we denoted the right hand side above by  ${\mathcal B}_1$.   
Below, once again using the shift law, 
 we make ${\mathcal B}_1$ uniformly indexed, i.e. with all sums 
having the same index range. Then we can easily verify that 
\bd
\begin{array}{ll}
{\mathcal B}_1=&
\sum\limits_{i=2}^{M^{x}-1}\,\,
\sum\limits_{j=2}^{M^{y}-1}h^{x}_{i}h^{y}_{j}v_{i-1,j-1,k}v_{ijk}
+\sum\limits_{i=2}^{M^{x}-1}\,\, 
\sum\limits_{j=2}^{M^{y}-1}h^{x}_{i}h^{y}_{j}v_{i,j-1,k}v_{i-1,j,k}\\
&+\sum\limits_{i=2}^{M^{x}-1}\,\, 
\sum\limits_{j=2}^{M^{y}-1}h^{x}_{i}h^{y}_{j}v_{i-1,j,k}v_{i,j-1,k}
+\sum\limits_{i=2}^{M^{x}-1}\,\, 
\sum\limits_{j=2}^{M^{y}-1}h^{x}_{i}h^{y}_{j}v_{i,j,k}v_{i-1,j-1,k}\\
&+\sum\limits_{i=2}^{M^{x}-1}\,\, 
\sum\limits_{j=2}^{M^{y}-1}h^{x}_{i}h^{y}_{j}v_{ijk}^{2}
+\sum\limits_{i=2}^{M^{x}-1}\,\, 
\sum\limits_{j=2}^{M^{y}-1}h^{x}_{i}h^{y}_{j}v_{i-1,j,k}^{2}\\ 
&+\sum\limits_{i=2}^{M^{x}-1}\, \, 
\sum\limits_{j=2}^{M^{y}-1}h^{x}_{i}h^{y}_{j}v_{i,j-1,k}^{2}
+\sum\limits_{i=2}^{M^{x}-1}\,\, 
\sum\limits_{j=2}^{M^{y}-1}h^{x}_{i}h^{y}_{j}v_{i-1,j-1,k}^{2} \\
&=
\sum\limits_{i=2}^{M^{x}-1}\,\, \sum\limits_{j=2}^{M^{y}-1}h^{x}_{i}h^{y}_{j}
\left( 
v_{i-1,j-1,k}^{2} +2v_{i-1,j-1,k}v_{ijk}+ v_{ijk}^{2} +v_{i,j-1,k}^{2}+\right.\\
& \left .\qquad  +2v_{i,j-1,k}v_{i-1,j,k}+v_{i-1,j,k}^{2}
\right)\\
&\ge
-\frac{1}{4}\left(
\sum\limits_{i=1}^{M^{x}-2}\,\,
\sum\limits_{j=1}^{M^{y}-2}h^{x}_{i+1}h^{y}_{j+1}v_{ijk}^{2} +
\sum\limits_{i=2}^{M^{x}-1}\,\, 
\sum\limits_{j=2}^{M^{y}-1}h^{x}_{i}h^{y}_{j}v_{ijk}^{2}
\right.\\ &
\qquad \quad\left.+
\sum\limits_{i=1}^{M^{x}-2}\,\, 
\sum\limits_{j=2}^{M^{y}-1}h^{x}_{i+1}h^{y}_{j}v_{ijk}^{2}+
\sum\limits_{i=2}^{M^{x}-1}\,\, 
\sum\limits_{j=1}^{M^{y}-2}h^{x}_{i}h^{y}_{j+1}v_{ijk}^{2}
\right).
\end{array}
\ed
Now, recalling the definition of ${\mathcal A}_1$ and 
using the bound for  ${\mathcal B}_1$ iteratively, 
we can derive the following chain of estimates 
\bd
\begin{array}{l}
\frac{1}{16}\sum\limits_{i=1}^{M^{x}-1}\sum\limits_{j=1}^{M^{y}-1}
\left(
12\newhbar^{x}_{i}\newhbar^{y}_{j}v_{ijk}^{2} +
h^{x}_{i}h^{y}_{j}v_{i-1,j-1,k}v_{ijk} +
h^{x}_{i+1}h^{y}_{j}v_{i+1,j-1,k}v_{ijk} 
\right.\\  \left. \qquad\qquad +
h^{x}_{i}h^{y}_{j+1}v_{i-1,j+1,k}v_{ijk} 
h^{x}_{i+1}h^{y}_{j+j}v_{i+1,j+1,k}v_{ijk} \right) 
\\ 
 \ge 
\frac{1}{16}\sum\limits_{i=1}^{M^{x}-1}\sum\limits_{j=1}^{M^{y}-1}
\left(
11\newhbar^{x}_{i}\newhbar^{y}_{j}v_{ijk}^{2} -
\frac{1}{4}\left(
\sum\limits_{i=1}^{M^{x}-2}\sum\limits_{j=1}^{M^{y}-2}h^{x}_{i+1}h^{y}_{j+1}v_{ijk}^{2} 
\right.\right.\\
\qquad  +\left.\left.
\sum\limits_{i=2}^{M^{x}-1}\sum\limits_{j=2}^{M^{y}-1}h^{x}_{i}h^{y}_{j}v_{ijk}^{2}
\sum\limits_{i=1}^{M^{x}-2}\sum\limits_{j=2}^{M^{y}-1}h^{x}_{i+1}h^{y}_{j}v_{ijk}^{2}+
\sum\limits_{i=2}^{M^{x}-1}\sum\limits_{j=1}^{M^{y}-2}h^{x}_{i}h^{y}_{j+1}v_{ijk}^{2}
\right) \right)
\\
 \ge
\frac{10}{16}\sum\limits_{i=1}^{M^{x}-1}\sum\limits_{j=1}^{M^{y}-1}\newhbar^{x}_{i}\newhbar^{y}_{j}v_{ijk}^{2} +
\frac{1}{16}\left(
\sum\limits_{i=1}^{M^{x}-1}\sum\limits_{j=1}^{M^{y}-1} 
\frac{h^{x}_{i}+h^{x}_{i+1}}{4}\frac{h^{y}_{j}+h^{y}_{j+1}}{4}v_{ijk}^{2}
-\sum\limits_{i=1}^{M^{x}-2}\sum\limits_{j=1}^{M^{y}-2}\frac{h^{x}_{i+1}h^{y}_{j+1}}{4}v_{ijk}^{2}
\right.
\\ \left. \qquad -
\sum\limits_{i=2}^{M^{x}-1}\sum\limits_{j=2}^{M^{y}-1}\frac{h^{x}_{i}h^{y}_{j}}{4}v_{ijk}^{2}
-\sum\limits_{i=1}^{M^{x}-2}\sum\limits_{j=2}^{M^{y}-1}\frac{h^{x}_{i+1}h^{y}_{j}}{4}v_{ijk}^{2}
-\sum\limits_{i=2}^{M^{x}-1}\sum\limits_{j=1}^{M^{y}-2}\frac{h^{x}_{i}h^{y}_{j+1}}{4}v_{ijk}^{2}
\right)
\\
=
\frac{10}{16}\sum\limits_{i=1}^{M^{x}-1}\sum\limits_{j=1}^{M^{y}-1}\newhbar^{x}_{i}\newhbar^{y}_{j}v_{ijk}^{2} +
\frac{1}{64}v_{ijk}^{2}\left(
h^{x}_{i}h^{y}_{j} \left(\sum\limits_{i=1}^{M^{x}-1}\sum\limits_{j=1}^{M^{y}-1}1- \sum\limits_{i=2}^{M^{x}-1}\sum\limits_{j=2}^{M^{y}-1}1\right)
\right. \\ 
\left.\qquad +
h^{x}_{i}h^{y}_{j+1} \left(\sum\limits_{i=1}^{M^{x}-1}\sum\limits_{j=1}^{M^{y}-1}1- \sum\limits_{i=2}^{M^{x}-1}\sum\limits_{j=1}^{M^{y}-2}1\right)
+h^{x}_{i+1}h^{y}_{j} \left(\sum\limits_{i=1}^{M^{x}-1}\sum\limits_{j=1}^{M^{y}-1}1- \sum\limits_{i=1}^{M^{x}-2}\sum\limits_{j=2}^{M^{y}-1}1\right)
\right. \\ 
\left. \qquad
+h^{x}_{i+1}h^{y}_{j+1} \left(\sum\limits_{i=1}^{M^{x}-1}\sum\limits_{j=1}^{M^{y}-1}1- \sum\limits_{i=1}^{M^{x}-2}\sum\limits_{j=1}^{M^{y}-2}1\right)
\right)
\ge 
\frac{10}{16}\sum\limits_{i=1}^{M^{x}-1}\sum\limits_{j=1}^{M^{y}-1}\newhbar^{x}_{i}\newhbar^{y}_{j}v_{ijk}^{2}, 
\end{array}
\ed
where in the last step we used that all the differences of the sums are 
positive. Note in particular the role of the coefficient 12 in the 
central differencing term 
and the chain of split
 in this term. Finally, recalling the definition of $(\mu_{xy}v,v]_{z}$, 
we multiply the above estimate by $\newhbar^{z}_{k}$ and sum over $k$ to 
obtain. 
\bd
\begin{array}{l}
(\mu_{xy}v,v]_{z} \ge \frac{10}{16} \sum\limits_{i=1}^{M^{x}-1}\sum\limits_{j=1}^{M^{y}-1}\sum\limits_{k=1}^{M^{z}}
\newhbar^{x}_{i}\newhbar^{y}_{j}\newhbar^{z}_{k}v_{ijk}^{2}
=\frac{10}{16}\Vert v \vert ]_{z}^{2} \ge \frac{1}{2}\Vert v \vert ]_{z}^{2}.
\end{array}
\ed
This completes the proof of the first assertion $i)$ of the lemma. 
The other two estimates are derived by similar 
 calculus, alternating the relevant sub- and super-indices, and 
therefore are omitted. 
\end{proof}
In the general case of $d$ dimensions we can see that 
the coefficient will become $\frac{3\cdot 2^{d-2}-1}{2^{d}}$. 
The general ratio above is linked to the coefficient of the central term 
in the finite difference case \eqref{FinDiffCentral}. 


\begin{lemma}\label{lem32}
Let $v$ be a mesh function on $\bar{\Omega}^{h}$ that vanishes on $\partial \Omega^{h}$, then
$$ \Vert v \Vert^{2} \le \frac{1}{3} \vert v \vert_{1,h}^{2}.
$$
\end{lemma}
\begin{proof} Using the definitions of the divided differences and 
following the notation, the desired result is obtained 
through the successive estimates below 
\bd
\begin{array}{l}
\Vert v \Vert ^{2} = \sum\limits_{i=1}^{M^{x}-1}\sum\limits_{j=1}^{M^{y}-1}\sum\limits_{k=1}^{M^{z}-1}\newhbar^{x}_{i}\newhbar^{y}_{j}\newhbar^{z}_{k}v_{ijk}^{2}
= \frac{1}{3}\left(
\sum\limits_{i=1}^{M^{x}-1}\sum\limits_{j=1}^{M^{y}-1}\sum\limits_{k=1}^{M^{z}-1}\newhbar^{x}_{i}\newhbar^{y}_{j}\newhbar^{z}_{k}\vert\sum\limits_{m=1}^{i}h^{x}_{m}\Delta_{x}^{-}v_{mjk}\vert^{2}
\right. \\ \left. 
+\sum\limits_{i=1}^{M^{x}-1}\sum\limits_{j=1}^{M^{y}-1}\sum\limits_{k=1}^{M^{z}-1}\newhbar^{x}_{i}\newhbar^{y}_{j}\newhbar^{z}_{k}\vert\sum\limits_{m=1}^{j}h^{y}_{m}\Delta_{y}^{-}v_{imk}\vert^{2}
+\sum\limits_{i=1}^{M^{x}-1}\sum\limits_{j=1}^{M^{y}-1}\sum\limits_{k=1}^{M^{z}-1}\newhbar^{x}_{i}\newhbar^{y}_{j}\newhbar^{z}_{k}\vert\sum\limits_{m=1}^{k}h^{z}_{m}\Delta_{z}^{-}v_{ijm}\vert^{2}
\right)
 \\ 
\qquad \quad\le
\frac{1}{3}\left(
\sum\limits_{i=1}^{M^{x}-1}\sum\limits_{j=1}^{M^{y}-1}\sum\limits_{k=1}^{M^{z}-1}\newhbar^{x}_{i}\newhbar^{y}_{j}\newhbar^{z}_{k}
\left(
\left(\sum\limits_{m=1}^{i}h^{x}_{m}\right)\left(\sum\limits_{m=1}^{i}h^{x}_{m}\vert \Delta^{-}_{x} v_{mjk} \vert^{2}\right)
\right. \right. \\ \left. \left. \qquad \quad 
+\left(\sum\limits_{m=1}^{i}h^{x}_{m}\right)\left(\sum\limits_{m=1}^{j}h^{y}_{m}\vert \Delta^{-}_{y} v_{imk} \vert^{2}\right)
+\left(\sum\limits_{m=1}^{i}h^{x}_{m}\right)\left(\sum\limits_{m=1}^{k}h^{z}_{m}\vert \Delta^{-}_{z} v_{ijm} \vert^{2}\right)
\right)
\right)
 \\ \qquad \quad =
\frac{1}{3}
\left( \sum\limits_{m=1}^{M^{x}}\sum\limits_{j=1}^{M^{y}-1}\sum\limits_{k=1}^{M^{z}-1}h^{x}_{m}\vert\Delta^{-}_{x}v_{mjk} \vert^{2}\newhbar^{y}_{j}\newhbar^{z}_{k}\right)
\left(\sum\limits_{i=1}^{M^{x}-1}\newhbar^{x}_{i}\sum\limits_{m=1}^{i}h^{x}_{m}\right)
\\ \qquad\quad +
\frac{1}{3}
\left( \sum\limits_{i=1}^{M^{x}-1}\sum\limits_{m=1}^{M^{y}}\sum\limits_{k=1}^{M^{z}-1}h^{y}_{m}\vert\Delta^{-}_{y}v_{imk} \vert^{2}\newhbar^{x}_{i}\newhbar^{z}_{k}\right)
\left(\sum\limits_{j=1}^{M^{y}-1}\newhbar^{y}_{j}\sum\limits_{m=1}^{j}h^{y}_{m}\right)
\\ \qquad \quad+
\frac{1}{3}
\left( \sum\limits_{i=1}^{M^{x}-1}\sum\limits_{j=1}^{M^{y}-1}\sum\limits_{m=1}^{M^{z}}h^{z}_{m}\vert\Delta^{-}_{z}v_{ijm} \vert^{2}\newhbar^{x}_{i}\newhbar^{y}_{j}\right)
\left(\sum\limits_{k=1}^{M^{z}-1}\newhbar^{z}_{k}\sum\limits_{m=1}^{k}h^{z}_{m}\right)
\\ \qquad\quad \le
\frac{1}{3}\left( \Vert \Delta^{-}_{x}v \vert]_{x}^{2} + \Vert \Delta^{-}_{y}v \vert]_{y}^{2} + \Vert \Delta^{-}_{z}v \vert]_{z }^{2} \right)
= \frac{1}{3}\vert v \vert_{1,h}^{2} \le \frac{1}{2}\vert v \vert_{1,h}^{2}.
\end{array}
\ed
\end{proof}
In the general case of $d$ dimensions the coefficient $1/3$ above, becomes $1/d$.

Based on these estimates 
we can prove the counterparts of Theorems 3.1 and 3.2 in 
\cite{Suli:91} in three (as well as higher) dimensions.
\begin{theorem}\label{thm31}
Let $L^{h}v=-(\Delta^{+}_{x}\Delta^{-}_{x}\mu_{yz}+\Delta^{+}_{y}\Delta^{-}_{y}\mu_{xz}+\Delta^{+}_{z}\Delta^{-}_{z}\mu_{xy})v$, then
$$ \Vert v \Vert_{1,h} \le \frac{32}{15}\Vert L^{h}v \Vert_{-1,h}. $$
\end{theorem}
\begin{proof}
Evidently, we have the identities 
\bd
(-\Delta^{+}_{x}w,v)=(w,\Delta^{-}_{x}v]_{x},\quad (-\Delta^{+}_{y}w,v)=(w,\Delta^{-}_{y}v]_{y},\quad (-\Delta^{+}_{z}w,v)=(w,\Delta^{-}_{z}v]_{z}. 
\ed
Therefore, using Lemmas \ref{lem31} and \ref{lem32} yields 
\bd
\begin{array}{ll}
(L^{h}v,v)&=(-(\Delta^{+}_{x}\Delta^{-}_{x}\mu_{yz}+\Delta^{+}_{y}\Delta^{-}_{y}\mu_{xz}+\Delta^{+}_{z}\Delta^{-}_{z}\mu_{xy})v,v)
\\ &=
(\Delta^{-}_{x}\mu_{yz}v,\Delta^{-}_{x}v]_{x}+(\Delta^{-}_{y}\mu_{xz}v,\Delta^{-}_{y}v]_{y}+(\Delta^{-}_{z}\mu_{xy}v,\Delta^{-}_{z}v]_{z}\\
& \ge \frac{5}{8} \left(\Vert \Delta^{-}_{x}v \vert ]_{x} + \Vert \Delta^{-}_{y}v \vert ]_{y} + \Vert \Delta^{-}_{z}v \vert ]_{z}\right)=
\frac{5}{8} \vert v \vert_{1,h}^{2} \ge \frac{15}{32} \Vert v \Vert_{1,h}^{2}.
\end{array}
\ed
Thus, by the definition of $\Vert \cdot \Vert_{-1,h}$ we obtain,
\bd
\Vert v \Vert_{1,h} \le \frac{32}{15}\Vert L^{h}v \Vert_{-1,h}.
\ed
\end{proof}

In $d$ dimensions following the same procedure we obtain 
\bd
\Vert v \Vert_{1,h} \le \frac{2^{d}\left(1+d\right)}{d\left( 3 \cdot 2^{d-2} - 1\right) } \Vert L^{h} v \Vert_{-1,h}.
\ed

\begin{theorem}\label{thm32}
If $f\in H^{\sigma}(\Omega)$, $\sigma>-1/2$,  
then the convolution $T_{111}$ is continuous and 
the equation \eqref{eqfds} has a unique solution $u^{h}$. Further, 
\bd \Vert u^{h} \Vert_{1,h} \le \frac{32}{30} \Vert T_{111}f \Vert_{-1,h}.\ed
\end{theorem}
\begin{proof}
Follows directly from Eq. \eqref{eqfds} and Theorem \ref{thm31}.
\end{proof}
In $d$ dimensions we will obtain 
\bd
\Vert u^{h} \Vert_{1,h} \le \frac{2^{d}\left(1+d\right)}{2d\left( 3 \cdot 2^{d-2} - 1\right)} \Vert T_{1\ldots 1}f \Vert_{-1,h}.
\ed

\section{Convergence analysis}
In this section we derive convergence rate for the proposed 
finite volume scheme. 
Most of the results in here hold true for the corresponding finite 
difference- and finite element-schemes as well. In the  convergence 
rate proofs, 
we shall use the following classical 
result:
 
\begin{theorem}\label{thm42}
Let $W$ be a Banach space and $W_{1}$ a normed linear space. Let  
$A:W\to W_{1}$ be a compact linear operator
and let $S_{1}:W\to\mathbb{R}$ and $S_{2}:W\to\mathbb{R}$ 
denote two bounded sublinear functionals
(i.e. 
$S_{i}(\alpha u + \beta v)\le \vert \alpha\vert S_{i}(u)+
\vert\beta \vert S_{i}(v)$ for $,\alpha,\,\beta \in \mathbb{R}$ and $,u,\,v \in W$).
Further, assume that there exists a constant $C_{0}$ such that,
$$
\Vert v\Vert_{W} \le C_{0} \left(\Vert Av \Vert_{W_{1}} + S_{2}(v) \right)\qquad 
\forall {v\in W}, 
$$
and that $\mathrm{Ker}(S_{2}) \subset \mathrm{Ker}(S_{1})$.
Then
\begin{enumerate}[i)]
\item $P:=\mathrm{Ker}(S_{2})$ is a finite dimensional vector space,
\item there exists a constant $C_{1}$ 
such that $\inf\limits_{p\in P}\Vert v-p \Vert_{W} \le C_{1}S_{2}(v)$ 
\quad $\forall {v\in W}$,
\item there exists a constant $C_{2}$ such that $S_{1}(v) \le C_{2}S_{2}(v)$.
\end{enumerate}
\end{theorem}
\begin{proof}
Follows directly from Theorem \ref{thmTartar} (see Appendix) by taking $E=W$,  
$E_{0},\, E_{1},\, F=W_{1}$, $S_{1}=L$ and $S_{2}=A_{1}$.
\end{proof}
 
\begin{theorem}\label{thm41}
If $u\in H^{1+\sigma}(\Omega)$, 
$1/2<\sigma\le 2$, 
then
\bd
\Vert u - u^{h}\Vert_{1,h} \le Ch^{\sigma}\vert u \vert_{H^{1+\sigma}(\Omega)},
\ed
where $h=\max_{i,j,k}(h^{x}_{i},h^{y}_{j},h^{z}_{k})$ and the constant 
$C>0$ does not depend on $u$ and the 
discretization parameters. 
\end{theorem}
This is an optimal result corresponding to a finite element 
approach without a quadrature 
(gives an $L_2$-estimate of order  ${\mathcal O}(h^{\sigma +1})$). 
With the 
same regularity, i.e.  $u\in H^{1+\sigma}(\Omega)$, the corresponding 
$L_2$-estimate for 
the finite element method with quadrature rule, 
and the finite difference method, 
would have a lower convergence rate of order 
${\mathcal O}(h^{\sigma +1/2})$.

\begin{proof} 
For a cuboid 
$\omega =\Pi_{i=1}^d\omega_i :=  \Pi_{i=1}^d[a_i, b_i]\subset \mathbb{R}^{d}$ and 
a  $d$-dimensional multi-index $\alpha:= (\alpha_1,\ldots,\alpha_d)$, 
for $i=1,\ldots, d$, we use the notation 
$\alpha^{i}:=(0,\ldots, 0, \alpha_{i},\ldots, 0)$ (only the $i$-th coordinate 
is non--zero) and set 
$\omega_{-i}:= \omega\setminus\omega_{i}$. 
Further we denote by $x_{-i}$ 
the $(d-1)$ dimensional vector 
$x_{-i}:=(x_1,\ldots, x_{i-1}, x_{i+1},\ldots x_d)$. 
Then  we define $H^{\alpha}(\omega)$, the anisotropic Sobolev space,  
that consists of all functions $u \in L_{2}(\omega)$ such that 

\bd
\norm{ u }_{H^{\alpha^{i}}(\omega)} = 
\Big(\int_{\omega_{-i}} \vert u(x_{-i})\vert_{H^{\alpha_{i}}(\omega_{i})}^{2} \, dx 
\Big)^{1/2} < \infty. 
\ed 

$H^{\alpha}(\omega)$ is a Banach space with the norm,
\bd
\Vert u \Vert_{H^{\alpha}(\omega)} = \left(\Vert u \Vert_{L^{2}(\omega)}^{2} + \vert u \vert_{H^{\alpha}(\omega)}^{2} \right)^{1/2} =
\left(\Vert u \Vert_{L^{2}(\omega)}^{2} + \sum\limits_{i=1}^{d} \vert u \vert_{H^{\alpha^{i}}(\omega)}^{2} \right)^{1/2}, 
\ed
see, e.g. \cite{J.L.Lions:61}. Further, if   
we denote the global error function 
by $z=u-u^{h}$, then as $T_{111}f=L^{h}u^{h}$ and $f=-\Delta u$  we have,
\bd
L^{h}z = \left( T_{111}\frac{\partial^{2}u}{\partial x^{2}} - \Delta^{+}_{x} \Delta^{-}_{x} \mu_{yz} u \right)
+ \left( T_{111}\frac{\partial^{2}u}{\partial y^{2}} - \Delta^{+}_{y} \Delta^{-}_{y} \mu_{xz} u \right)
+ \left( T_{111}\frac{\partial^{2}u}{\partial z^{2}} - \Delta^{+}_{z} \Delta^{-}_{z} \mu_{xy} u \right). 
\ed
We can easily verify that 
\bd
\begin{array}{ll}
(T_{111} \frac{\partial^{2}u}{\partial x^{2}})_{ijk} &= \frac{1}{2} \frac{1}{\newhbar^{x}_{i}\newhbar^{y}_{j}\newhbar^{z}_{k}} \left( \chi_{ijk} \ast \frac{\partial^{2}u}{\partial x^{2}}\right)(x_{ijk})\\
&= \frac{1}{2} \frac{1}{\newhbar^{x}_{i}\newhbar^{y}_{j}\newhbar^{z}_{k}} 
\int\limits_{z_{k-1/2}}^{z_{k+1/2}}\int\limits_{y_{j-1/2}}^{y_{j+1/2}} 
\frac{\partial u}{\partial x}(x_{i+1/2},y,z) - \frac{\partial u}{\partial x}(x_{i-1/2},y,z) \ud y\ud x\\
&= \frac{1}{2}\Delta^{+}_{x} (T^{-}_{011} \frac{\partial u}{\partial x})_{ijk},
\end{array}
\ed
where,
\bd
(T^{-}_{011} w)_{ijk} = \frac{1}{\newhbar^{y}_{j} \newhbar^{z}_{k}} \int\limits_{z_{k-1/2}}^{z_{k+1/2}}\int\limits_{y_{j-1/2}}^{y_{j+1/2}} w(x_{i-1/2},y,z) \ud y \ud z.
\ed
$( T_{111}\frac{\partial^{2}u}{\partial y^{2}})_{ijk}$, and $( T_{111}\frac{\partial^{2}u}{\partial y^{2}})_{ijk}$
are treated in analogous fashion, e.g. 
\bd
\begin{array}{ll}
(T^{-}_{101} w)_{ijk} & = \frac{1}{\newhbar^{x}_{i} \newhbar^{z}_{k}} \int\limits_{z_{k-1/2}}^{z_{k+1/2}}\int\limits_{x_{i-1/2}}^{x_{i+1/2}} w(x,y_{j-1/2},z) \ud x \ud z 
(T^{-}_{101} w)_{ijk} \\
&= \frac{1}{\newhbar^{x}_{i} \newhbar^{y}_{j}} \int\limits_{y_{j-1/2}}^{y_{j+1/2}}\int\limits_{x_{i-1/2}}^{x_{i+1/2}} w(x,y,z_{k-1/2}) \ud x \ud y.
\end{array}
\ed
This gives us 
\be\label{eq42}
\left\{
\begin{array}{rcl}
L^{h}z &  =  &\Delta^{+}_{x} \eta_{1} + \Delta^{+}_{y} \eta_{2} + \Delta^{+}_{z} \eta_{3} \qquad \mathrm{in}\quad\Omega^{h}, \\
z & = &  0\qquad\mathrm{on}\quad\partial \Omega^{h},
\end{array}
\right.
\ee
with
\bd
\begin{array}{rcl}
\eta_{1} & = & \frac{1}{2} T^{-}_{011} \frac{\partial u}{\partial x} - \Delta^{-}_{x}\mu_{yz}u, \\
\eta_{2} & = & \frac{1}{2} T^{-}_{101} \frac{\partial u}{\partial y} - \Delta^{-}_{y}\mu_{xz}u, \\
\eta_{3} & = & \frac{1}{2} T^{-}_{110} \frac{\partial u}{\partial z} - \Delta^{-}_{z}\mu_{xy}u.
\end{array}
\ed
Now from Eq. \eqref{eq42} and Theorem \ref{thm31} we can derive 
\bd
\Vert z \Vert_{1,h} \le \frac{32}{15} \Vert \Delta^{+}_{x} \eta_{1} + \Delta^{+}_{y} \eta_{2} + \Delta^{+}_{z} \eta_{3} \Vert_{-1,h}. 
\ed
We can also show that for certain mesh functions 
(e.g. shape regular)  defined on 
$\bar{\Omega}^{h}$ and vanishing on $\partial \Omega_{h}$ we have
\mbox{
$(-\Delta^{+}_{(\cdot)}w,g) = (w,\Delta^{-}_{(\cdot)} g]$}. Hence 
\bd
\begin{array}{ll}
\vert (\Delta^{+}_{x} \eta_{1} + \Delta^{+}_{y} \eta_{2} + \Delta^{+}_{z} \eta_{3},w) \vert
& = \vert (\eta_{1},\Delta^{-}_{x}w]_{x} + (\eta_{2},\Delta^{-}_{y}w]_{y} + (\eta_{3},\Delta^{-}_{z}w]_{z} \vert
\\ & \le 
\vert \Vert \eta_{1} \vert]_{x} \Vert \Delta^{-}_{x}w\vert]_{x} + \Vert\eta_{2} \vert ]_{y} \Vert \Delta^{-}_{y}w\vert]_{y} + \Vert\eta_{3}\vert]_{z} \Vert\Delta^{-}_{z}w \vert ]_{z} \vert 
\\ & \le
\left( \Vert \eta_{1} \vert]_{x} + \Vert \eta_{2} \vert]_{y} + \Vert \eta_{3} \vert]_{z}   \right) \Vert w\Vert_{1,h}.
\end{array}
\ed
Thus, by the definition of the dual norm, we get 
\bd
\begin{array}{l}
\vert (\Delta^{+}_{x} \eta_{1} + \Delta^{+}_{y} \eta_{2} + \Delta^{+}_{z} \eta_{3},w) \vert \Vert w \Vert_{1,h}^{-1}
\le \Vert \eta_{1} \vert]_{x} + \Vert \eta_{2} \vert]_{y} + \Vert \eta_{3} \vert]_{z}   \\
\Vert (\Delta^{+}_{x} \eta_{1} + \Delta^{+}_{y} \eta_{2} + \Delta^{+}_{z} \eta_{3},w) \Vert_{-1,h} 
\le \Vert \eta_{1} \vert]_{x} + \Vert \eta_{2} \vert]_{y} + \Vert \eta_{3} \vert]_{z}.
\end{array}
\ed 
Therefore,
\be\label{eq43}
\Vert u-u^{h} \Vert_{1,h}  \le \frac{32}{15}\left( \Vert \eta_{1} \vert]_{x} + \Vert \eta_{2} \vert]_{y} + \Vert \eta_{3} \vert]_{z} \right).
\ee
Now we have to bound the right--hand side of  \eqref{eq43}. 
Here we only consider the  \mbox{
$\eta_{1}$-term}  as the other
two can be treated in the same way. To this end, 
for a fixed $x$ let $I_{yz}w(x,\cdot,\cdot)$ denote the piecewise interpolant of $w(x,\cdot,\cdot)$ on the mesh 
$\bar{\Omega}_{yz}^{h}$ and
\bd
\begin{array}{ll}
&(\mu_{yz} u )(x,y_{j},z_{k}) = \frac{1}{16} \frac{1}{\newhbar^{y}_{j}\newhbar^{z}_{k}} \left( 
h^{y}_{j} h^{z}_{k} u(x,y_{j-1},z_{k-1}) + h^{y}_{j+1} h^{z}_{k} u(x,y_{j+1},z_{k-1})\right. \\ 
& \qquad\left. + 12\newhbar^{y}_{j} \newhbar^{z}_{k} u(x,y_{j},z_{k})
+h^{y}_{j} h^{z}_{k+1} u(x,y_{j-1},z_{k+1}) + h^{y}_{j+1} h^{z}_{k+1} u(x,y_{j+1},z_{k+1})
\right),  
\end{array}
\ed
then 
\bd
(\mu_{yz}u)(x,y_{j},z_{k}) = \frac{1}{\newhbar^{y}_{j} \newhbar^{z}_{k}} \int\limits_{y_{j-1/2}}^{y_{j+1/2}}\int\limits_{z_{k-1/2}}^{z_{k+1/2}} (I_{yz} u)(x,y,z)\ud z \ud y.
\ed
Further, using

\bd
\begin{array}{ll}
(\mu_{yz}u)_{ijk} - (\mu_{yz}u)_{i-1,j,k} &= \int\limits_{x_{i-1}}^{x_{i}} \frac{\partial}{\partial x} (\mu_{yz}u)(x,y_{j},x_{k})\ud x\\
&=\int\limits_{x_{i-1}}^{x_{i}} \frac{\partial}{\partial x} \frac{1}{\newhbar^{y}_{j} \newhbar^{z}_{k}} \int\limits_{y_{j-1/2}}^{y_{j+1/2}}\int\limits_{z_{k-1/2}}^{z_{k+1/2}} (I_{yz}u)(x,y,z) \ud z \ud y\ud x\\
&=\frac{1}{\newhbar^{y}_{j} \newhbar^{z}_{k}}\int\limits_{x_{i-1}}^{x_{i}}  \int\limits_{y_{j-1/2}}^{y_{j+1/2}}\int\limits_{z_{k-1/2}}^{z_{k+1/2}} \frac{\partial}{\partial x} (I_{yz}u)(x,y,z) \ud z \ud y\ud x\\
&=\frac{1}{\newhbar^{y}_{j} \newhbar^{z}_{k}}\int\limits_{x_{i-1}}^{x_{i}}  \int\limits_{y_{j-1/2}}^{y_{j+1/2}}\int\limits_{z_{k-1/2}}^{z_{k+1/2}}  I_{yz}\left (\frac{\partial u}{\partial x}\right)(x,y,z) \ud z \ud y\ud x, 
\end{array}
\ed
we can write $(\eta_{1})_{ijk}$ as 
\bd
(\eta_{1})_{ijk} = \frac{1}{\newhbar^{x}_{i}\newhbar^{y}_{j}\newhbar^{z}_{k}} \int\limits_{x_{i-1}}^{x_{i}}  \int\limits_{y_{j-1/2}}^{y_{j+1/2}}\int\limits_{z_{k-1/2}}^{z_{k+1/2}}
\left( \frac{1}{2}\frac{\partial u}{\partial x}(x_{i-1/2},y,z) - T_{011}\left(\frac{\partial u}{\partial x}\right)(x,y,z) \right)
\ud z \ud y \ud x .
\ed
Now we split $(\eta_{1})_{ijk}$ into a sum of four terms: 
\bd
\begin{array}{rcl}
(\eta_{11})_{ijk} & = & \frac{1}{\newhbar^{x}_{i}\newhbar^{y}_{j}\newhbar^{z}_{k}} \int\limits_{x_{i-1}}^{x_{i}}  \int\limits_{y_{j}}^{y_{j+1/2}}\int\limits_{z_{k}}^{z_{k+1/2}}
\left( \frac{1}{2}\frac{\partial u}{\partial x}(x_{i-1/2},y,z) - T_{011}\left(\frac{\partial u}{\partial x}\right)(x,y,z) \right)
\ud z \ud y \ud x ,\\
(\eta_{12})_{ijk}& = &\frac{1}{\newhbar^{x}_{i}\newhbar^{y}_{j}\newhbar^{z}_{k}} \int\limits_{x_{i-1}}^{x_{i}}  \int\limits_{y_{j}}^{y_{j+1/2}}\int\limits_{z_{k-1/2}}^{z_{k}}
\left( \frac{1}{2}\frac{\partial u}{\partial x}(x_{i-1/2},y,z) - T_{011}\left(\frac{\partial u}{\partial x}\right)(x,y,z) \right)
\ud z \ud y \ud x, \\
(\eta_{13})_{ijk}& = &\frac{1}{\newhbar^{x}_{i}\newhbar^{y}_{j}\newhbar^{z}_{k}} \int\limits_{x_{i-1}}^{x_{i}}  \int\limits_{y_{j-1/2}}^{y_{j}}\int\limits_{z_{k}}^{z_{k+1/2}}
\left( \frac{1}{2}\frac{\partial u}{\partial x}(x_{i-1/2},y,z) - T_{011}\left(\frac{\partial u}{\partial x}\right)(x,y,z) \right)
\ud z \ud y \ud x, \\
(\eta_{14})_{ijk}& =& \frac{1}{\newhbar^{x}_{i}\newhbar^{y}_{j}\newhbar^{z}_{k}} \int\limits_{x_{i-1}}^{x_{i}}  \int\limits_{y_{j-1/2}}^{y_{j}}\int\limits_{z_{k-1/2}}^{z_{k}}
\left( \frac{1}{2}\frac{\partial u}{\partial x}(x_{i-1/2},y,z) - T_{011}\left(\frac{\partial u}{\partial x}\right)(x,y,z) \right)
\ud z \ud y \ud x. 
\end{array}
\ed
Thus, to estimate $\eta_{1}$ it suffices to estimate 
$\eta_{11}$, $\eta_{13}$, $\eta_{13}$ and $\eta_{14}$.
Here, we only show how to estimate $\eta_{11}$ as the other three terms 
will follow in the same way. 
We introduce the change of variables 
\bd
x=x_{i-1/2}+sh^{x}_{i},~~~-\frac{1}{2} \le s \le \frac{1}{2}; \quad 
y=y_{j}+th^{y}_{j+1},~~~ 0 \le t \le 1; \quad
z=z_{k}+rh^{z}_{k+1},~~~ 0 \le r \le 1, 
\ed
and define
\bd
\tilde{v}(s,t,r) := h^{x}_{i}\frac{\partial u}{\partial x}(x(s),y(t),z(r)).
\ed
This gives us 
\bd
(\eta_{11})_{ijk} = \frac{h^{y}_{j+1}h^{z}_{k+1}}{h^{x}_{i}\newhbar^{y}_{j}\newhbar^{z}_{k}} \tilde{\eta}_{11},
\ed
with
\bd
\tilde{\eta}_{11} = \int\limits_{-1/2}^{1/2}\int\limits_{0}^{1/2}\int\limits_{0}^{1/2}
\frac{1}{2}\tilde{v}(0,t,r) - \left(\tilde{v}(s,0,0)(1-t-r)+ \tilde{v}(s,1,0)t +\tilde{v}(s,0,1)r\right)
\ud r \ud t \ud s. 
\ed
Note that $\tilde{v}(0,t,r) = \frac{\partial u}{\partial x}(x_{i-1/2},y,z)$ and 
$$
\tilde{v}(s,0,0)(1-t-r)+ \tilde{v}(s,1,0)t +\tilde{v}(s,0,1)r=(I_{yz}\frac{\partial u}{\partial x})(x,y,z). 
$$
Hence we treat $\tilde{\eta}_{11}$ as a linear functional with the argument $\tilde{v}$ defined on $H^{\sigma}(\tilde{\omega})$
$\sigma>1/2$, 
where $\tilde{\omega}=(-\frac{1}{2},\frac{1}{2})\times(0,1)\times(0,1)$.
Note that $\sigma >1/2$ is due to the fact that all 
$\eta$ components, defined by $T_{111}$, are convolutions with 
the characteristic function $\xi_{i,j,k}$. Since 
$\xi\in H^\tau({\mathbb R}^3),\,\,\tau<1/2$, continuity requires 
$\sigma >1/2$. 
Notice further that, for a given $\tilde{v}$, $\tilde{\eta}_{11}$ is constant
and its value on the boundary is the same as anywhere inside the domain.
Therefore,  by the trace theorem we have 
\bd
\vert \tilde{\eta}_{11} \vert \le C \Vert \tilde{v} \Vert_{H^{\sigma}(\tilde{\omega})}, \qquad \sigma>1/2, 
\ed
and using Theorem \ref{thm42} with 
$W=H^{\sigma}(\tilde{\omega})$, $W_{1}=L_{2}(\tilde{\omega})$, 
$S_{1}=\vert \tilde{\eta}_{11}\vert$, \\
$S_{2}=\left(\vert \cdot \vert^2_{H^{\sigma,0,0}(\tilde{\omega)}} 
+\vert \cdot \vert^2_{H^{0,\sigma,0}(\tilde{\omega)}}
+\vert \cdot \vert^2_{H^{0,0,\sigma}(\tilde{\omega)}}
\right)^{1/2}$ and with $A:H^{\sigma}(\tilde{\omega})\to L_{2}(\tilde{\omega})$ 
being the compact embedding operator
we obtain
\bd
\vert \tilde{\eta}_{11}(\tilde{v}) \vert \le C \left(\vert \tilde{v} \vert_{H^{\sigma,0,0}(\tilde{\omega)}}^{2} 
+\vert \tilde{v} \vert_{H^{0,\sigma,0}(\tilde{\omega)}}^{2} +\vert \tilde{v} \vert_{H^{0,0,\sigma}(\tilde{\omega)}}^{2}\right)^{1/2}
\ed
for $\sigma>1/2.$ 
We let now $\omega^{++}_{ijk}= (x_{i-1},x_{i}) \times (y_{j},y_{j+1})\times(z_{k}, z_{k+1})$, then 
returning to the original variables we obtain

{\small
\bd
\vert \tilde{\eta}_{11} \vert^{2} \le C \left(
\frac{h^{x^{2}}_{i}}{{h^{x^{2\sigma}}_{i}}}{h^{x}_{i}h^{y}_{j+1}h^{z}_{k}}
\vert \frac{\partial u}{\partial x} \vert_{H^{\sigma,0,0}(\omega^{++}_{ijk})}^{2}
+
\frac{{h^{x^{2}}_{i}}{h^{y^{2\sigma}}_{j+1}}}{h^{x}_{i}h^{y}_{j+1}h^{z}_{k}}
\vert \frac{\partial u}{\partial x} \vert_{H^{0,\sigma,0}(\omega^{++}_{ijk})}^{2}
+
\frac{{h^{x^{2}}_{i}}{h^{z^{2\sigma}}_{k+1}}}{h^{x}_{i}h^{y}_{j+1}h^{z}_{k}}
\vert \frac{\partial u}{\partial x} \vert_{H^{0,0,\sigma}(\omega^{++}_{ijk})}^{2}
\right).
\ed
}
Thus
{\footnotesize
\bd
\vert (\eta_{11})_{ijk} \vert^{2} \le C \left(
\frac{h^{y}_{j+1}h^{z}_{k+1}
{h^{x^{2\sigma-1}}_{i}}}{{\newhbar^{y^{2}}_{j}}{\newhbar^{z^{2}}_{k}}} \vert \frac{\partial u}{\partial x}^2 \vert_{H^{\sigma,0,0}(\omega^{++}_{ijk})}
+
\frac{
{h^{y^{2\sigma+1}}_{j+1}}h^{z}_{k+1}}{h^{x}_{i}{\newhbar^{y^{2}}_{j}}{\newhbar^{z^{2}}_{k}}} \vert \frac{\partial u}{\partial x}^2 \vert_{H^{0,\sigma,0}(\omega^{++}_{ijk})}
+
\frac{h^{y}_{j+1}{h^{z^{2\sigma+1}}_{k+1}}}{h^{x}_{i}{\newhbar^{y^{2}}_{j}}{\newhbar^{z^{2}}_{k}}} \vert \frac{\partial u}{\partial x}^2 \vert_{H^{0,0,\sigma}(\omega^{++}_{ijk})}
\right). 
\ed
}
Similar estimates are derived for 
\bd
\begin{array}{lcl}
\omega^{+-}_{ijk} &=& (x_{i-1},x_{i}) \times (y_{j},y_{j+1})\times(z_{k-1}, z_{k}), \\
\omega^{-+}_{ijk} &=& (x_{i-1},x_{i}) \times (y_{j-1},y_{j})\times(z_{k}, z_{k+1}), \\
\omega^{--}_{ijk} &=& (x_{i-1},x_{i}) \times (y_{j-1},y_{j})\times(z_{k-1}, z_{k}),
\end{array}
\ed
leading to 
{\small
\bd
\begin{array}{l}
\vert (\eta_{12})_{ijk} \vert^{2}  \le C \left(
\frac{h^{y}_{j+1}h^{z}_{k}{h^{x^{2\sigma-1}}_{i}}}{{\newhbar^{y^{2}}_{j}}{\newhbar^{z^{2}}_{k}}} \vert \frac{\partial u}{\partial x}^2 \vert_{H^{\sigma,0,0}(\omega^{+-}_{ijk})}
+
\frac{{h^{y^{2\sigma+1}}_{j+1}}h^{z}_{k}}{h^{x}_{i}{\newhbar^{y^{2}}_{j}}{\newhbar^{z^{2}}_{k}}} \vert \frac{\partial u}{\partial x}^2 \vert_{H^{0,\sigma,0}(\omega^{+-}_{ijk})}
+
\frac{h^{y}_{j+1}{h^{z^{2\sigma+1}}_{k}}}{h^{x}_{i}{\newhbar^{y^{2}}_{j}}{\newhbar^{z^{2}}_{k}}} \vert \frac{\partial u}{\partial x}^2 \vert_{H^{0,0,\sigma}(\omega^{+-}_{ijk})}
\right)
\\
\vert (\eta_{13})_{ijk} \vert^{2}  \le  C \left(
\frac{h^{y}_{j}h^{z}_{k+1}{h^{x^{2\sigma-1}}_{i}}}{{\newhbar^{y^{2}}_{j}}{\newhbar^{z^{2}}_{k}}} \vert \frac{\partial u}{\partial x}^2 \vert_{H^{\sigma,0,0}(\omega^{-+}_{ijk})}
+
\frac{{h^{y^{2\sigma+1}}_{j}}h^{z}_{k+1}}{h^{x}_{i}{\newhbar^{y^{2}}_{j}}{\newhbar^{z^{2}}_{k}}} \vert \frac{\partial u}{\partial x}^2 \vert_{H^{0,\sigma,0}(\omega^{-+}_{ijk})}
+
\frac{h^{y}_{j}{h^{z^{2\sigma+1}}_{k+1}}}{h^{x}_{i}{\newhbar^{y^{2}}_{j}}{\newhbar^{z^{2}}_{k}}} \vert \frac{\partial u}{\partial x}^2 \vert_{H^{0,0,\sigma}(\omega^{-+}_{ijk})}
\right)
\\
\vert (\eta_{14})_{ijk} \vert^{2}  \le  C \left(
\frac{h^{y}_{j}h^{z}_{k}{h^{x^{2\sigma-1}}_{i}}}{{\newhbar^{y^{2}}_{j}}{\newhbar^{z^{2}}_{k}}} \vert \frac{\partial u}{\partial x}^2 \vert_{H^{\sigma,0,0}(\omega^{--}_{ijk})}
+
\frac{{h^{y^{2\sigma+1}}_{j}}h^{z}_{k}}{h^{x}_{i}{\newhbar^{y^{2}}_{j}}{\newhbar^{z^{2}}_{k}}} \vert \frac{\partial u}{\partial x}^2 \vert_{H^{0,\sigma,0}(\omega^{--}_{ijk})}
+
\frac{h^{y}_{j}{h^{z^{2\sigma+1}}_{k}}}{h^{x}_{i}{\newhbar^{y^{2}}_{j}}{\newhbar^{z^{2}}_{k}}} \vert \frac{\partial u}{\partial x}^2 \vert_{H^{0,0,\sigma}(\omega^{--}_{ijk})}
\right).
\end{array}
\ed }
Writing $h=\max_{i,j,k}(h^{x}_{i},h^{y}_{j},h^{z}_{k})$, 
by the super--additivity of the Sobolev norm on a family of 
disjoint Lebesgue measurable subsets of $\Omega$,
\be
\begin{array}{rcl}
\Vert \eta_{1} \vert]_{x}^{2} & \le & Ch^{2\sigma}\left( \vert \frac{\partial u}{\partial x} \vert_{H^{\sigma,0,0}(\Omega)}^{2} + \vert \frac{\partial u}{\partial x} \vert_{H^{0,\sigma,0}(\Omega)}^{2} + \vert \frac{\partial u}{\partial x} \vert_{H^{0,0,\sigma}(\Omega)}^{2}\right),\\
\Vert \eta_{2} \vert]_{y}^{2} & \le & Ch^{2\sigma}\left( \vert \frac{\partial u}{\partial y} \vert_{H^{\sigma,0,0}(\Omega)}^{2} + \vert \frac{\partial u}{\partial y} \vert_{H^{0,\sigma,0}(\Omega)}^{2} + \vert \frac{\partial u}{\partial y} \vert_{H^{0,0,\sigma}(\Omega)}^{2}\right),\\
\Vert \eta_{3} \vert]_{z}^{2} & \le & Ch^{2\sigma}\left( \vert \frac{\partial u}{\partial z} \vert_{H^{\sigma,0,0}(\Omega)}^{2} + \vert \frac{\partial u}{\partial z} \vert_{H^{0,\sigma,0}(\Omega)}^{2} + \vert \frac{\partial u}{\partial z} \vert_{H^{0,0,\sigma}(\Omega)}^{2}\right).
\end{array}
\ee  
All together we arrive at,
\bd
\Vert u-u^{h}\Vert_{1,h} \le Ch^{\sigma}\vert u \vert_{H^{1+\sigma}(\Omega)},
\ed
for $1/2< \sigma \le 2$. 
\end{proof}
From the above calculus we can see that the proof will also carry over to the 
$d$--dimensional case,  but then $C$ will depend on $d$.

In \cite{Suli:91} it is shown that on a two--dimensional quasi--uniform mesh (i.e. there is a constant $C_{\ast}$
such that $h:=\max_{i,j}(h^{x}_{i},h^{y}_{j}) \le C_{\ast}\min_{i,j}(h^{x}_{i},h^{y}_{j})$) 
the finite volume method of \mbox{Eq. \eqref{eqfvm}}
is almost optimally accurate
in the discrete (over the mesh points) maximum \mbox{
norm $\Vert \cdot \Vert_{\infty}$,}
i.e. for $u\in H^{1+\sigma}(\Omega), \frac{1}{2}<\sigma \le 2$ we have 
\bd
\Vert u - u^{h} \Vert_{\infty} \le Ch^{\sigma} \sqrt{\vert \log h \vert}\vert u\vert_{H^{1+\sigma}(\Omega)},
\ed
where $C$ depends on $C_{\ast}$. This does not hold in the three--dimensional case as, 
\\ 
$$
\Vert u \Vert_{L^{\infty}(\Omega)} \le C \Vert u \Vert_{W^{k}_{p}(\Omega)} ,\qquad k>n/p,
$$
here $k$ is the number of derivatives and $p$ is the parameter of 
the $L_{p}$-space ($\Omega$ should be Lipschitz,
as it is in our case). But $n=3$ and $p=2$, requires $k>
3/2$, and if we use the 
inverse estimate to go down half of a derivative to $H^{1}(\Omega)$, 
then we need to pay with half a power of $h$. Thus in three--dimensional case the result
is ${h}^{\sigma-1/2}$, rather than ${h}^{\sigma}\sqrt \vert \log h\vert$.

\section{Numerical example}
We implemented the finite volume scheme described by  equation 
 \eqref{eqfvm} according to the finite difference scheme
for the equation \eqref{eqfds} in a C++ program called FVM. The code  
is available from the
{URL: \texttt{http://www.math.chalmers.se/\~{}mohammad}}.  The implementation
is general and allows for any dimension of the problem, a user 
defined mesh (through an external
text file) and a user defined data function $f$.
The data function should be in an external dynamically linked library and can be 
parametrized. The user can provide the values of the parameters via a text file at execution.
Therefore the user is completely free to specify a data function.
Furthermore the program can compare the solution to a user defined
function.
Similarly this function is provided inside an external dynamically linked library 
and it can also be parametrized through a text file.

We use the \texttt{ uBLAS} Boost and \texttt{umfpack} libraries for matrix operations. This has the one consequence 
that the sparse solver collapses in the three dimensional case if we increase the mesh size
above $54$ points in all directions. In the two dimensional case we did not observe
any problems with the sparse solver. 
For multidimensional numerical integration we use the 
\texttt{Cuba} library \cite{Hah2007}.
We tested our code for a number of different functions based on the normal distribution density
and on mollifier functions. 
We define the shrunk to the unit cube Gaussian function
in $k$ dimensions as,
$$
u(\vec{x}) = \exp(-\sum\limits_{i=1}^{k}\frac{1}{\tan(\pi x_{i})^{2}})\mathbf{1}_{\mathrm{unit~cube}}(\vec{x})
$$
a mollifier function shrunk to the unit cube in $k$ dimensions as,
$$
u(\vec{x}) = \exp((1-4\Vert \vec{x} - (0.5,\ldots,0.5) \Vert_{2}^{2})^{-1})\mathbf{1}_{\mathrm{unit~cube}}(\vec{x})
$$
and a multidimensional Hicks--Henne sine bump function as,
$$
u(\vec{x}) = \left(\sin\left(2\pi\left(0.25 - \Vert \vec{x} \Vert^{2} \right) \right) \right)^{3}\mathbf{1}_{\{\vec{x} : \sum(x_{i} - 0.5)^{2} \le 0.25 \}}(\vec{x}).
$$
We considered the following 
as the difference of two functions
$G_{1}$ and $G_{2}$ for $\vec{x}$ in the unit cube,
$$
u(\vec{x}) = G_{1}(\vec{x}) - 3\cdot G_{2}(2\vec{x} - 0.5).
$$
$G_{1}$ and $G_{2}$ were both either a Gaussian, mollifier or Hicks--Henne sine bump.

The mesh points were randomly distributed in all dimensions.
We present graphs of $L_{2}$, $H_{1}$ and relative errors 
of our implementation in Figs. \ref{figFuncErrorsGauss}, \ref{figFuncErrorsMoll} and \ref{figFuncErrorsHicksHenne} .
\begin{figure}[!ht]
\begin{center}
\includegraphics[width=0.4\textwidth]{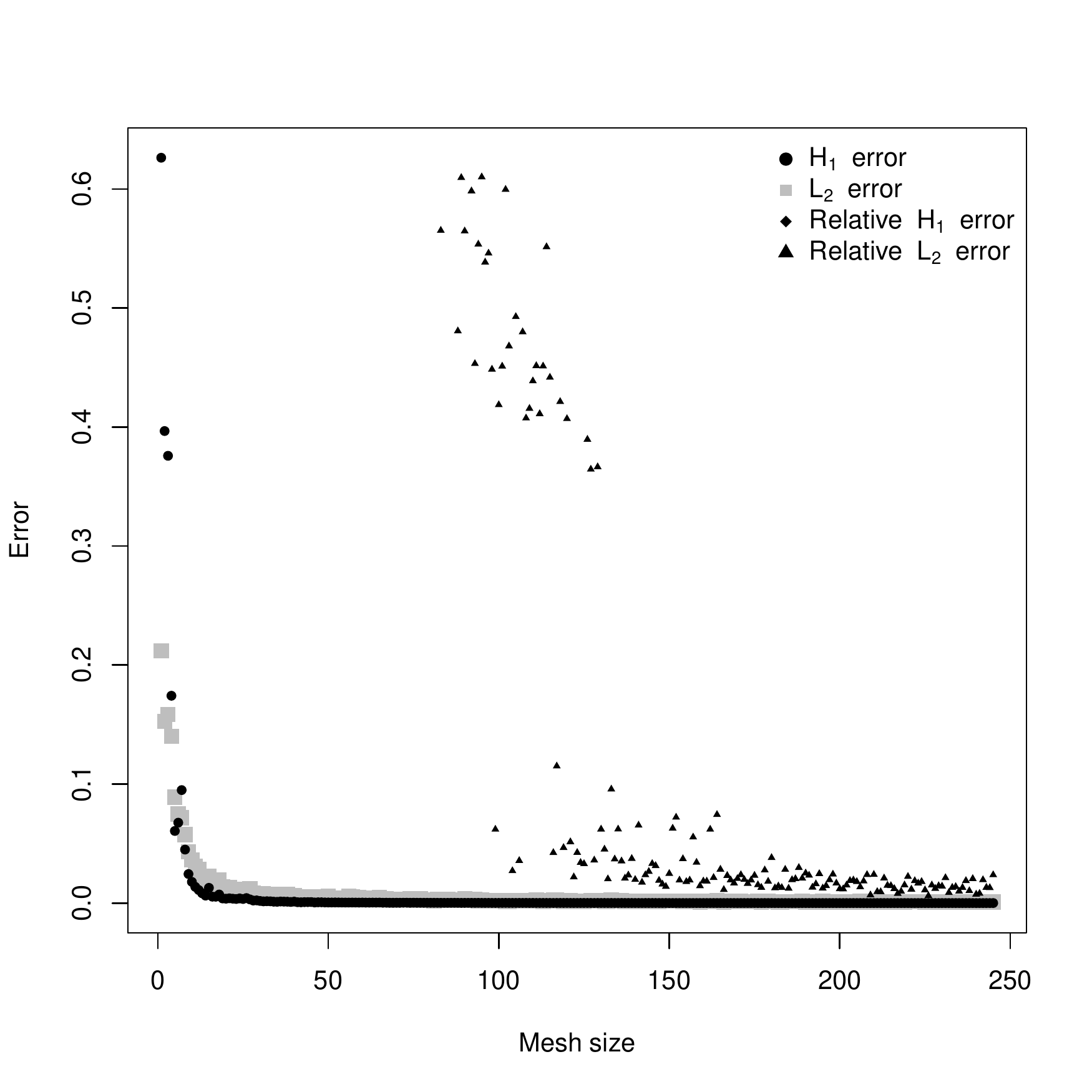}
\includegraphics[width=0.4\textwidth]{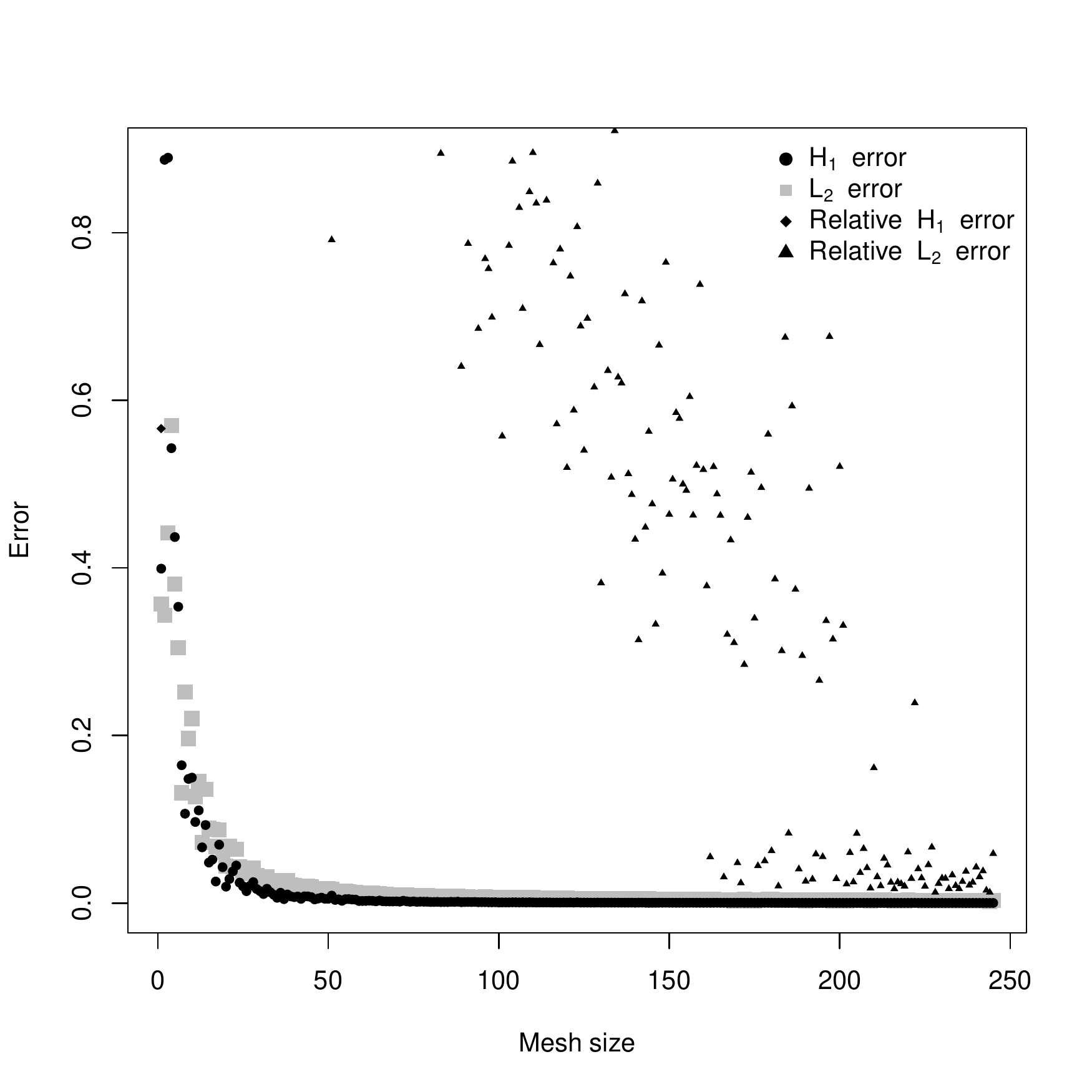} \\
\includegraphics[width=0.4\textwidth]{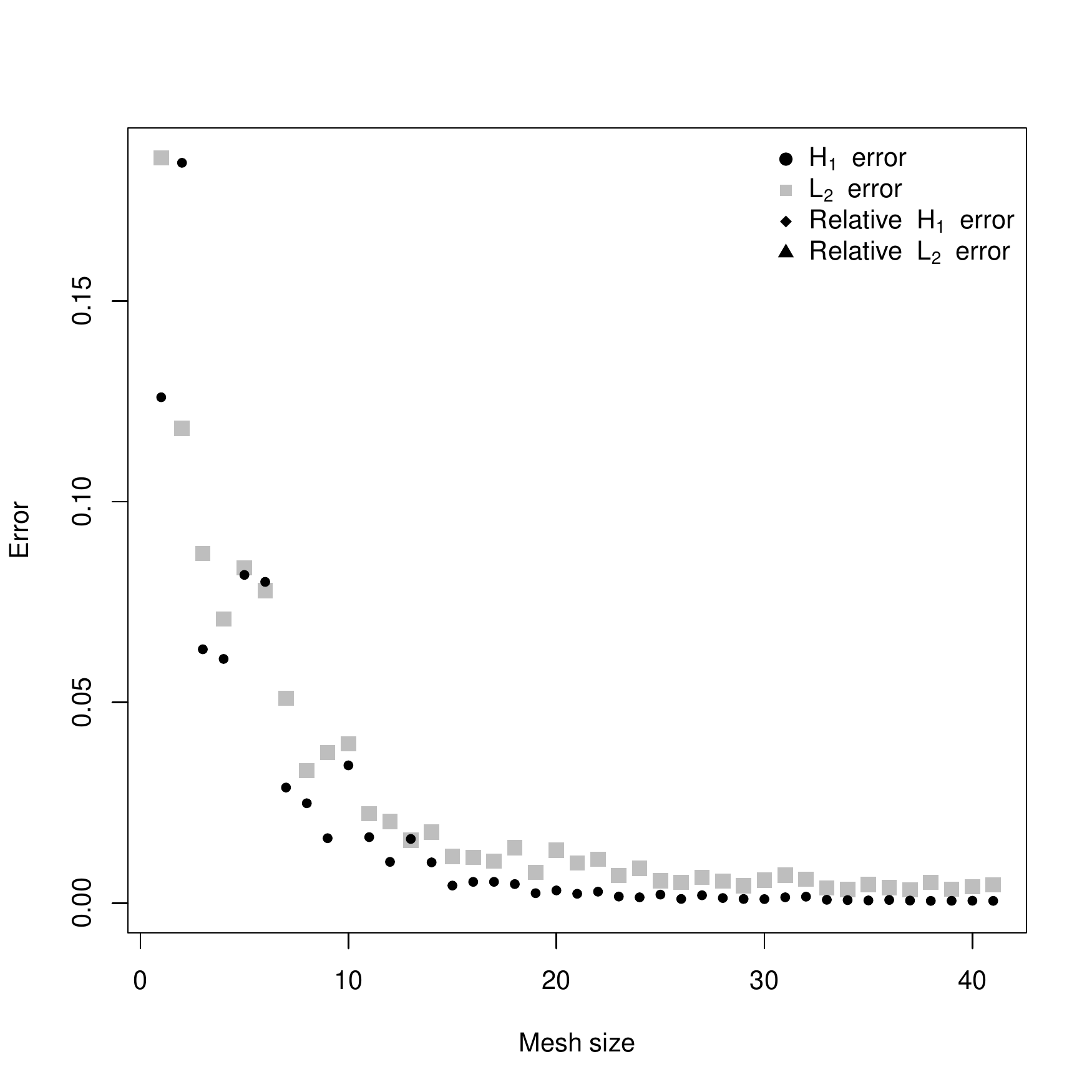} 
\includegraphics[width=0.4\textwidth]{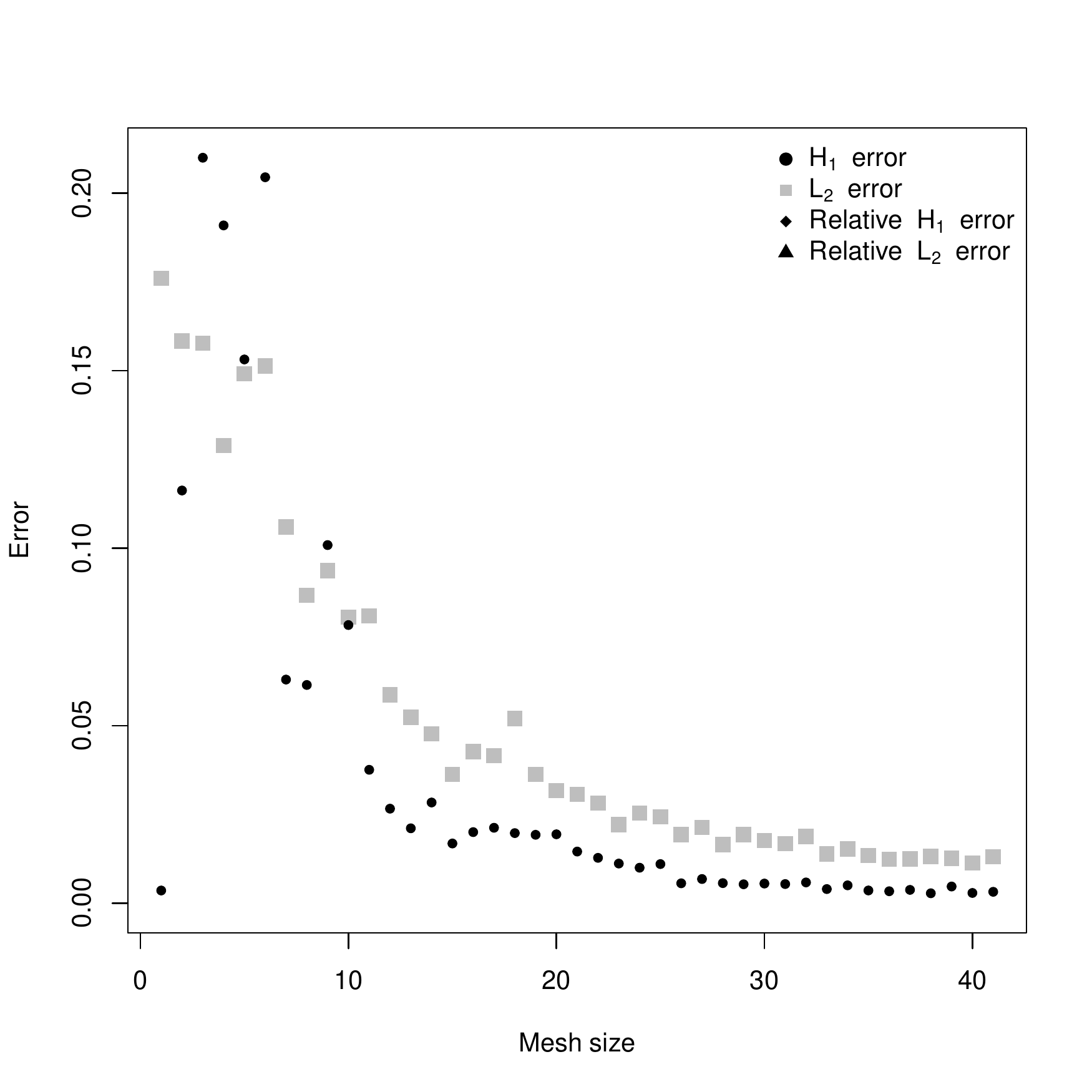} 
\caption{Errors for different functions.
Top: Gaussian function (left) and difference of two Gaussians (right) in two dimensions
and 
bottom: Gaussian function (left) and difference of two Gaussians (right) in three dimensions.
}
\label{figFuncErrorsGauss}
\end{center}
\end{figure}

\begin{figure}[!ht]
\begin{center}
\includegraphics[width=0.4\textwidth]{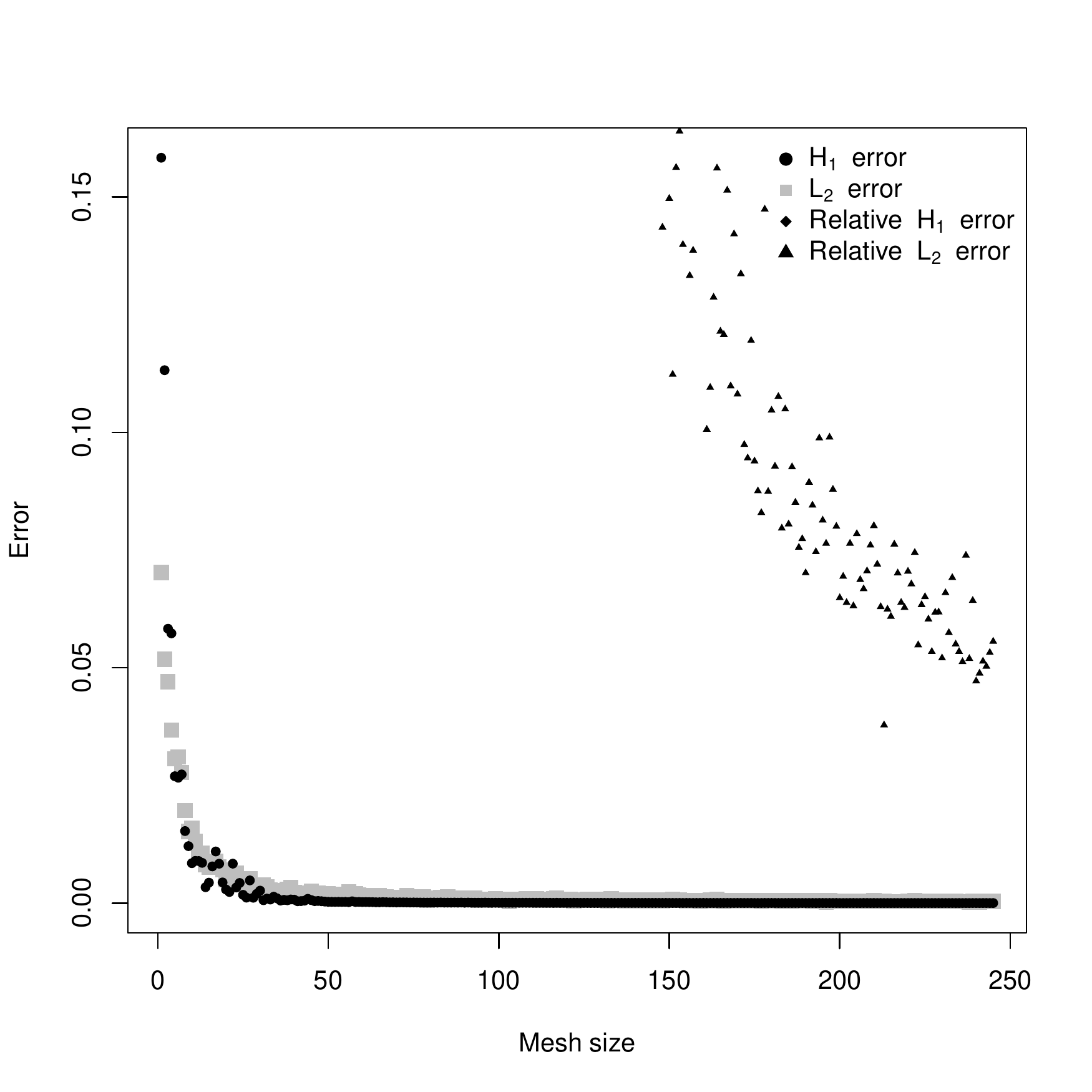}
\includegraphics[width=0.4\textwidth]{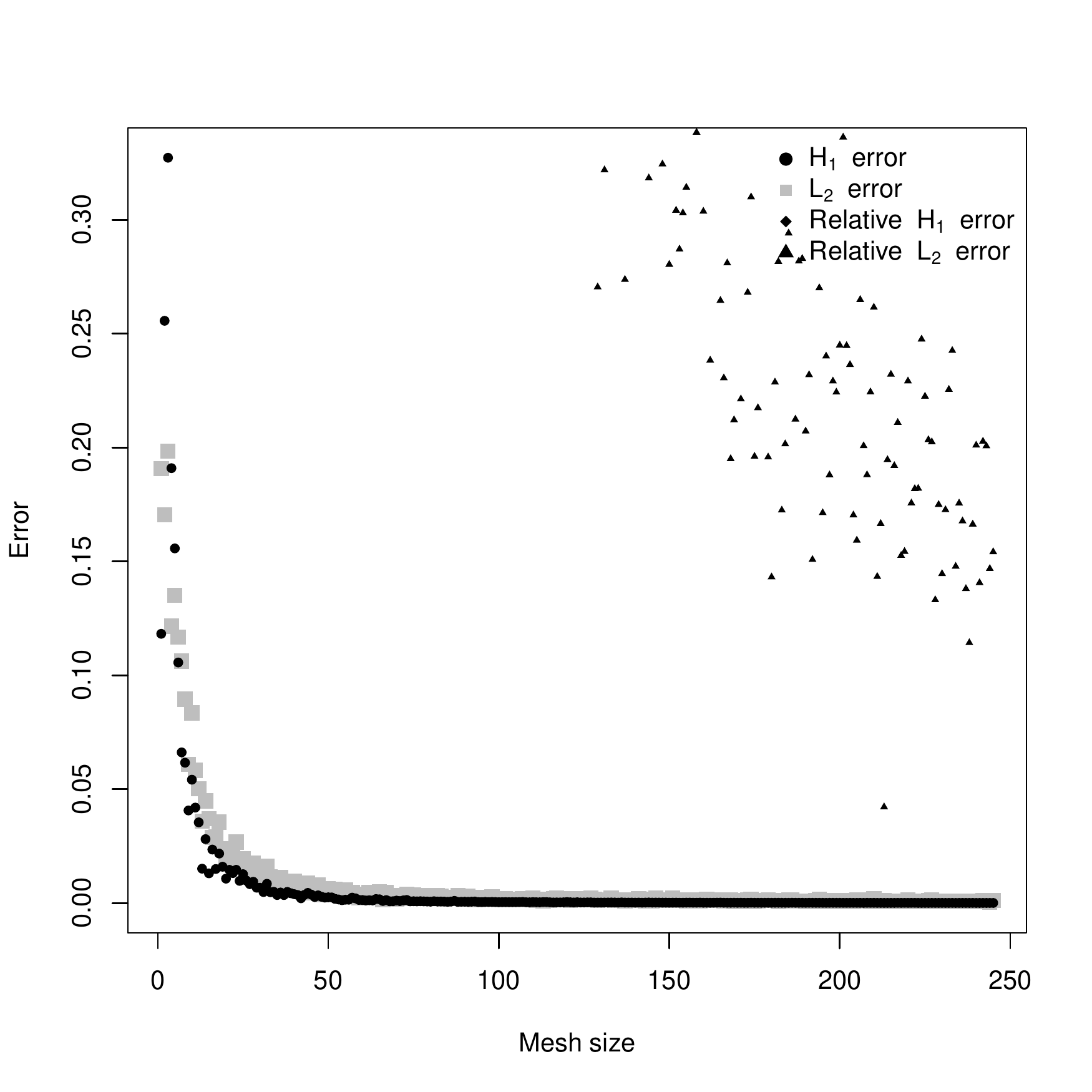}
\caption{Errors for different functions.
Left: mollifier function and right: difference of two mollifiers in two dimensions.
}
\label{figFuncErrorsMoll}
\end{center}
\end{figure}

\begin{figure}[!ht]
\begin{center}
\includegraphics[width=0.4\textwidth]{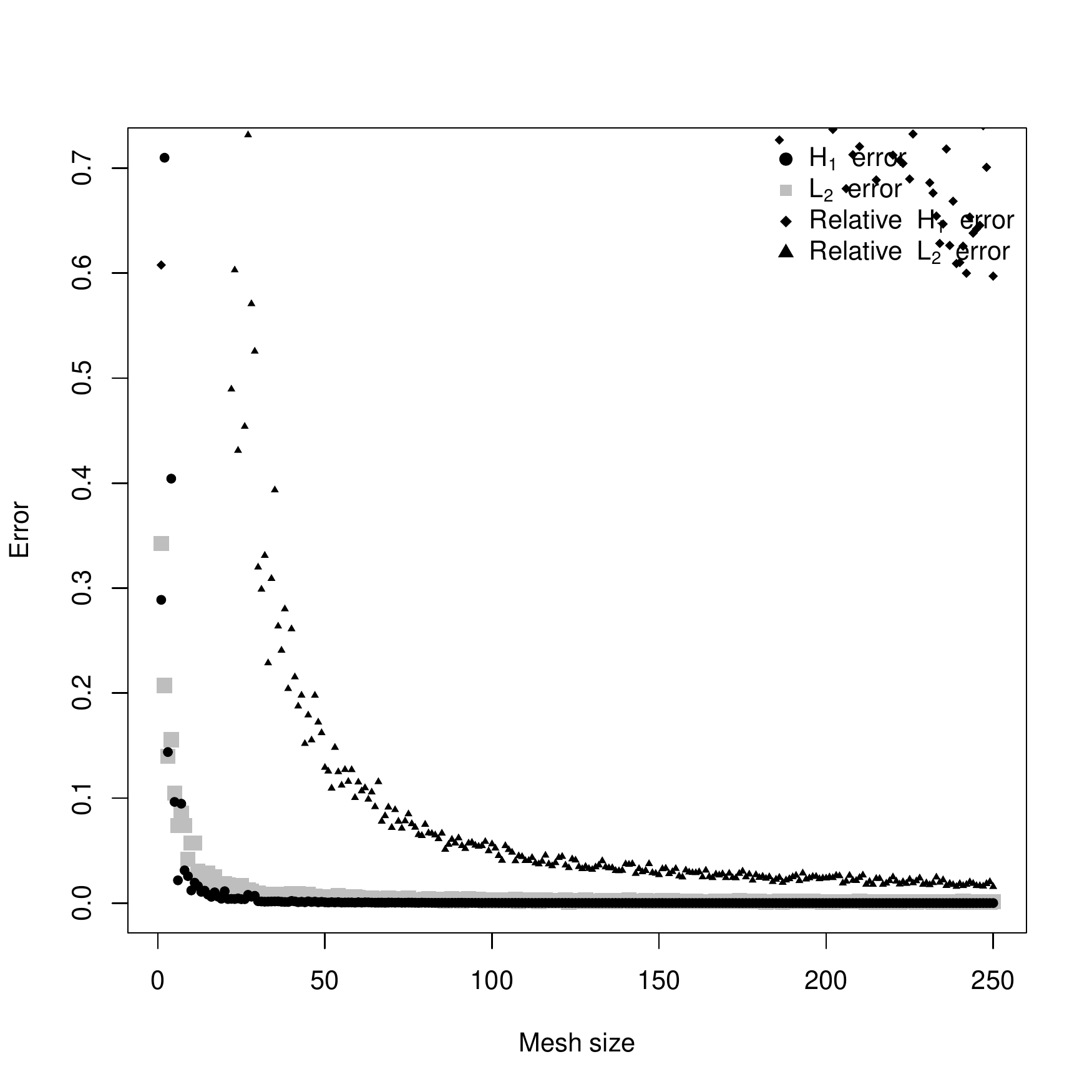}
\includegraphics[width=0.4\textwidth]{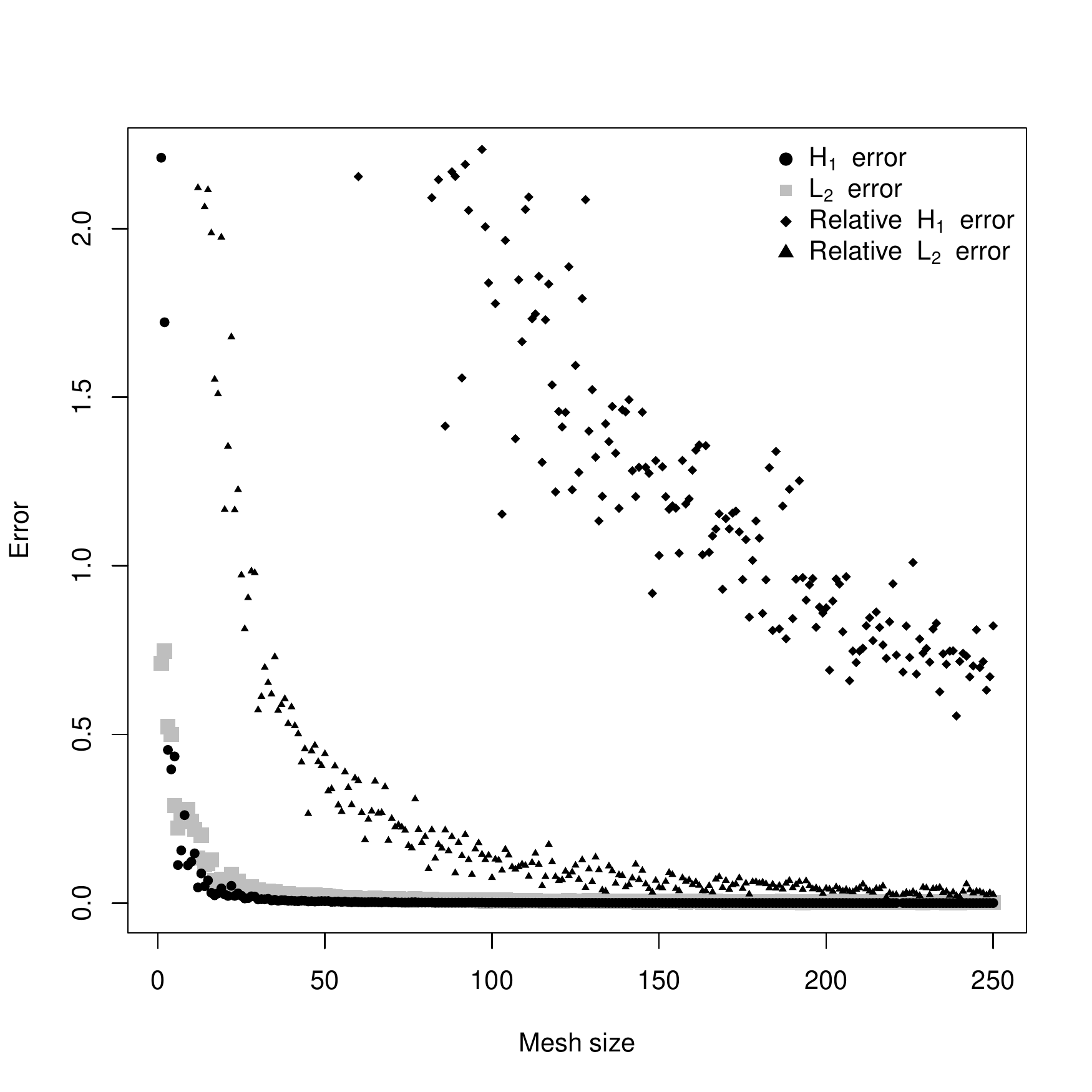} \\
\includegraphics[width=0.4\textwidth]{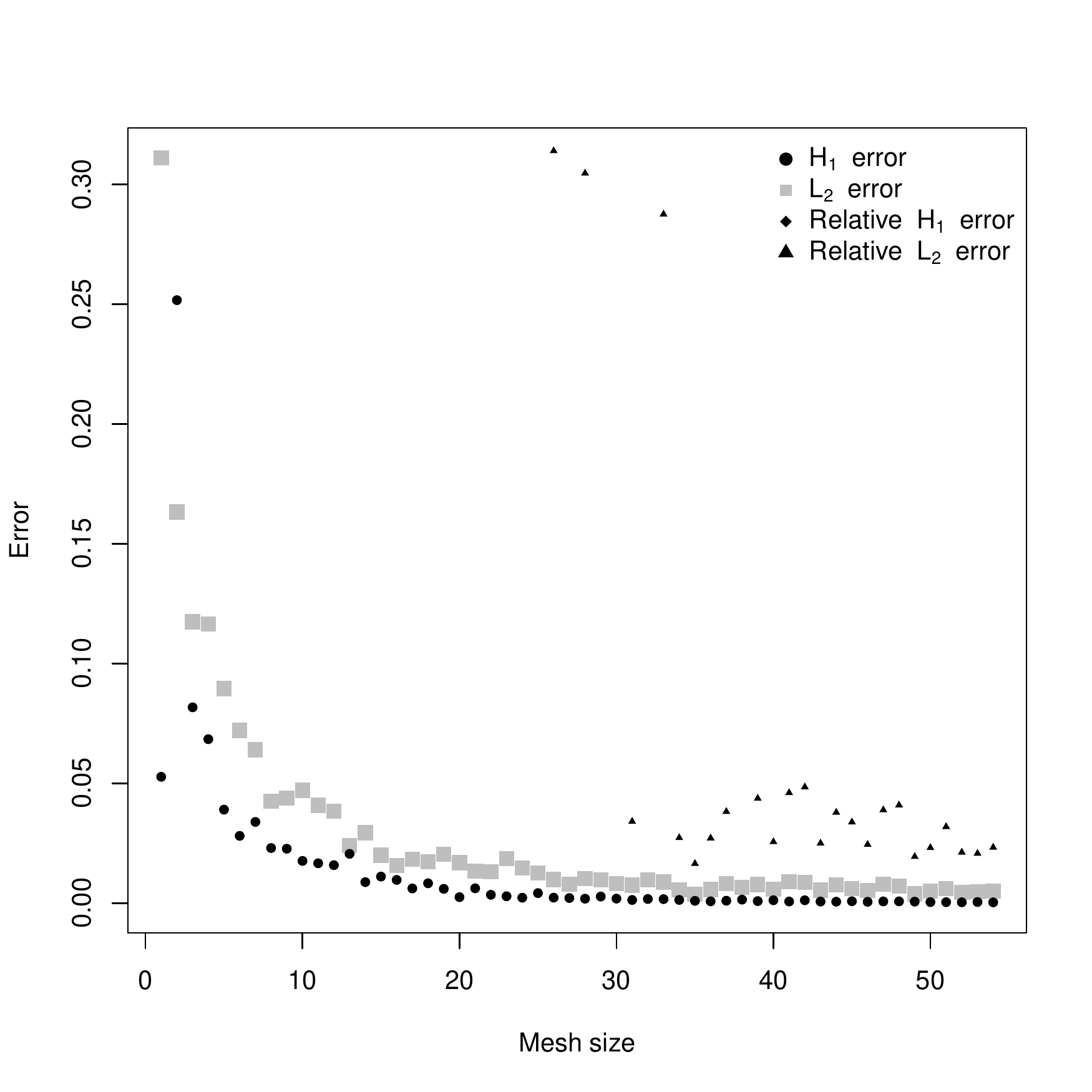} 
\includegraphics[width=0.4\textwidth]{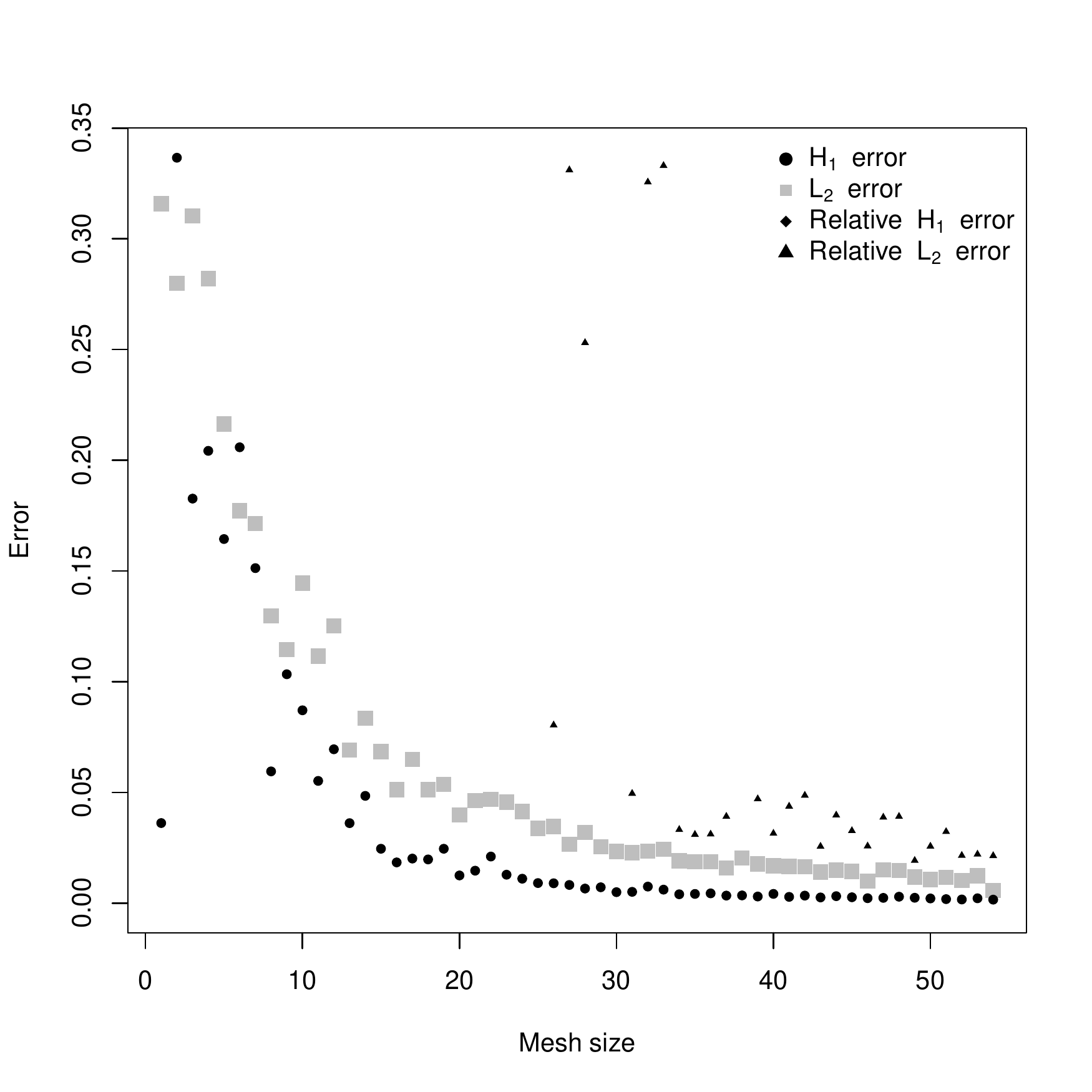} 
\caption{Errors for different functions.
Top: Hicks--Henne sine bump function (left) and difference of two Hicks--Henne sine bump functions (right) in two dimensions,
and
bottom: Hicks--Henne sine bump function (left) and difference of two Hicks--Henne sine bump functions (right) in three dimensions
}
\label{figFuncErrorsHicksHenne}
\end{center}
\end{figure}


\section*{Appendix}
\begin{theorem}\label{thmTartar}
Let $E$ be a Banach space  and let $E_{0}$, $E_{1}$ and $F$ be three  normed linear spaces,
$A_{0}$, $A_{1}$ and $L$ be linear continuous operators from $E$ into $E_{0}$, $E_{1}$ and $F$ respectively. If
\begin{enumerate}[i)]
\item \be\label{eqBoundgE} \Vert g \Vert_{E} = C_{0} \left(\Vert A_{0}g \Vert_{E_{0}} + \Vert A_{1}g \Vert_{E_{1}} \right),\ee
\item $Lg=0$ if $A_{1}g=0$, i.e. $\mathrm{Ker}(L) \subset \mathrm{Ker}(A_{1})$,
\item $A_{0}$ is compact, 
\end{enumerate}
then, there exists a constant $C$ such that,
\be 
\label{eqTartar} \forall_{g \in E},~~~~\Vert Lg \Vert_{F} \le C\Vert A_{1} g \Vert_{E_{1}}. \ee
\end{theorem}

\begin{proof}
This theorem is an unpublished lemma of Tartar, mentioned as an exercise in \cite{Ciarlet:78}
and cited in \cite{Suli:91}. Both of the works indicate that its proof can be found in
\cite{Brezzi.Marini:75}. The proof starts by noticing that $P:=\mathrm{Ker}(A_{1})$ 
is finite dimensional
however the argument for this in \cite{Brezzi.Marini:75} is that if weak sequential convergence
implies norm convergence then it indicates that $P$ is finite dimensional. This argument is not
clear however as due to Schur \cite{JSchur:1968} we have that in $l^{1}$, 
weak sequential convergence is
equivalent to norm convergence. Below we provide an alternative proof.

We will use the property that a unit ball is compact if and only if 
 the subspace is finite dimensional.
Let us take $g \in P = \mathrm{Ker}(A_{1}) \subset E$ and 
of course we have 
$$ \Vert g \Vert_{P} \equiv  \Vert g \Vert_{E} \le C_{0} \left(\Vert A_{0}g \Vert_{E_{0}} + \Vert A_{1}g \Vert_{E_{1}} \right) =
C_{0} \Vert A_{0}g \Vert_{E_{0}} $$
hence we can write 
$$ \forall_{g \in P}, \quad 
\Vert A_{0}g\Vert_{E_{0}} \ge C\Vert g \Vert_{E} \equiv C \Vert g \Vert_{P}.$$

Let us assume that the kernel of $A_{1}$, $P$ is infinite dimensional and then $P$ is not bounded in particular
not totally bounded and hence will not have a finite $\epsilon$--net, meaning,
$$ 
\exists_{\epsilon > 0},~~~~ \forall_{n}, ~~~~
\exists_{\stackrel{g_{1},\ldots,g_{n}\in P,}{\Vert g_{i} \Vert_{P} \le 1}},~~~~
\Vert g_{i} - g_{j} \Vert_{P} \ge \epsilon.
$$
We assumed that $A_{0}$ is compact so (denoting by $K_{E}$ the unit ball in $E$ and by $K_{0}$ unit ball in $E_{0}$) 
$$ A_{0} (K_{E}) \subseteq \Vert A_{0} \Vert K_{0},$$ 
due to
$$\forall_{x \in K_{E}}, \quad \Vert  A_{0}x \Vert \le \Vert A_{0} \Vert\Vert x\Vert_{E} . $$
With $P$ being infinite dimensional we can write,
$$\Vert A_{0}g_{i} - A_{0}g_{j} \Vert_{E_{0}} \ge  C\Vert g_{i} - g_{j}\Vert_{E} \equiv C\Vert g_{i} - g_{j}\Vert_{P} \ge C \epsilon.$$
This means that $A_{0}(K_{E})$ does not have a finite $\epsilon$--net, 
 so $A_{0}(K_{E})$ would not be relatively compact contradicting that
$A_{0}$ is compact. Hence $P$ must be finite dimensional.

After establishing that $\dim P < \infty$ one can follow the proof found in 
\cite{Brezzi.Marini:75} but, for the sake of completeness, we repeat it below. 

The proof of Eq. \eqref{eqTartar} will be done in two steps.
For all $g \in E$ we use the notation 
$Q(g):=\inf\limits_{p\in E}\Vert g-p \Vert_{E}$.
\begin{description}
\item[I] First we shall prove that there exists a constant $C_{1}$ such that 
\be\label{eqProofI} \forall_{g\in E},~~~~  Q(g) \le C_{1} \Vert A_{1} g \Vert_{E_{1}} . \ee
\item[II] Secondly we shall show that there exists a constant $C_{2}$ such that
\be\label{eqProofII} \forall_{g\in E},~~~~ \Vert Lg \Vert_{F} \le C_{2} Q(g) , \ee
giving Eq. \eqref{eqTartar}: $\Vert Lg \Vert_{F} \le C\Vert A_{1} g \Vert_{E_{1}}$.

\vskip 0.3cm 
\noindent 
Proof of {\bf I.}  We prove the inequality \eqref{eqProofI} by a 
contradiction argument: 
assume that there is a sequence $\{g_{n} \}\subset E$
such that $\Vert A_{1}g_{n}\Vert_{E_{1}}\to 0$ and $Q(g_{n})=1$, i.e.,
$$ \forall_{n}, ~~~\exists_{\{g_{n}\}}:\quad  Q(g_{n}) > n \Vert A_{1}g_{n}\Vert_{E_{1}}$$
and for convenience we can rescale $1>\frac{n}{Q(g_{n})}\Vert A_{1}g_{n}\Vert_{E_{1}}$, 
so we can take $Q(g_{n})=1$.

As $P$ is finite dimensional and totally bounded (hence compact) there exists a sequence $\tilde{g}_{n}=g_{n}-p_{n}$ such that,
$$\Vert \tilde{g}_{n} \Vert_{E} = Q(g_{n}) = \inf\limits_{p\in P}\Vert  g_{n} - p \Vert_{E} = \Vert g_{n} - p_{n}\Vert_{E}. $$
Therefore we have $\Vert A_{1}\tilde{g}_{n} \Vert_{E_{1}} = \Vert A_{1}g_{n}\Vert_{E_{1}} \to 0 $
as $A_{1}p_{n}=0$. Since the sequence $\{ \tilde{g}_{n} \}$ is bounded in $E$ ($\Vert \tilde{g}_{n} \Vert_{E}=1$ as $Q(g_{n})=1$)
it will contain a weakly convergent subsequence $\tilde{g}_{n_{k}} \rightharpoonup g^{\ast} \in E$ giving
$A_{0}\tilde{g}_{n_{k}} \stackrel{E_{0}}{\longrightarrow} A_{0}g^{\ast} $
and
$A_{1}g_{n_{k}} \stackrel{E_{1}}{\longrightarrow} A_{1}g^{\ast} $
implying
$A_{1}\tilde{g}_{n_{k}} \stackrel{E_{1}}{\longrightarrow} A_{1}g^{\ast}=0. $ 
Combining this and Eq. \eqref{eqBoundgE} we get $g_{n_{k}} \stackrel{E}{\longrightarrow} g^{\ast}$
as
$$\Vert g_{n_{k}} -g^{\ast}\Vert_{E} \le C_{0}\left(\Vert A_{0}g_{n_{k}} - A_{0} g^{\ast} \Vert_{E_{0}} + \Vert A_{1}g_{n_{k}} - A_{1} g^{\ast} \Vert_{E_{1}} \right) $$
giving $\inf\limits_{p\in P} \Vert \tilde{g}_{n_{k}} - p \Vert_{E} \le \Vert \tilde{g}_{n_{k}} - g^{\ast} \Vert_{E} \to 0$. 
But this contradicts $Q(g_{n_{k}})=1$ and so there exists a constant $C_{1}$ such that $Q(g) \le C_{1} \Vert A_{1} g\Vert_{E_{1}}$.

\vskip 0.3cm 
\noindent 
Proof of {\bf II.} We now turn to Eq. \eqref{eqProofII}. As we assumed $L$ is continuous and by the assumption of the theorem  $Lp=0$ we
have, 
$$ \Vert Lg\Vert_{F} = \Vert Lg -Lp \Vert_{F} \le C_{2} \Vert g-p \Vert_{E}.
$$
Taking $\inf$ over $p\in P$ on both sides gives,
$$ \Vert Lg\Vert_{F} \le  C_{2} \inf\limits_{p\in P}  \Vert g-p \Vert_{E} = C_{2} Q(g) \le C_{1}C_{2} \Vert A_{1}g \Vert_{E_{1}},$$
as desired.
\end{description}
\end{proof}

\vskip 0.3cm 
\noindent 
{\bf Conclusion.} We construct and analyze 
  a finite volume method 
for Poisson's equation, using a quasi-uniform mesh, in 
the three dimensional cube 
$\Omega =(0,1)\times(0,1)\times (0,1)$. We derive both stability and 
convergence estimates.  
The convergence rates are optimal in an $L_2$-setting, 
whereas the  $L_\infty$ error estimates, which are optimal in 2D, are  
sub-optimal in 3D. 
This generalizes the 
two-dimensional result by S\"uli, \cite{Suli:91} to three dimensions.
We show that the 
underlying theory for the two-dimensional case, 
studied by Grisvard in \cite{Grisvard:85}, is extendable to three dimensions 
(with some draw-back for $L_\infty$ error estimate). We also include a 
corrected proof of a classical result, cited in \cite{Suli:91}, and used in 
convergence analysis. Finally we have implemented the scheme in the 
C++ environment, for a general $k$-dimensional unit cube,  
and for {\sl Gaussians},  {\sl mollifier} and 
multidimensional 
{\sl Hicks-Henne sine bump} functions. The implementations are justifying the 
convergence rates both in $L_2$- and $H^1$- norms. The Figures 2-4 
are showing the 
absolute and relative errors. 

\vskip 0.5cm 
\section*{Acknowledgments}
We would like to thank Wojciech Bartoszek (Gda\'nsk University of Technology) 
for his help in the proof of Theorem \ref{thmTartar}.

\providecommand{\bysame}{\leavevmode\hbox to3em{\hrulefill}\thinspace}

\end{document}